\documentclass{scrartcl}
\usepackage[top=20truemm,bottom=30truemm,left=30truemm,right=30truemm]{geometry}
\usepackage{amsmath,amssymb}
\usepackage{mathtools}
\usepackage[dvipdfmx]{graphicx} 
\usepackage{amsthm}
\usepackage{bm}

\theoremstyle{plain}
\newtheorem{thm}{Theorem}[section]
\newtheorem{lemma}[thm]{Lemma}
\newtheorem{prop}[thm]{Proposition}

\theoremstyle{remark}

\theoremstyle{definition}

\usepackage{xcolor}
\definecolor{cyan20}{cmyk}{.2,0,0,0}

\newcommand{\tn}[1]{{\color{black} #1}}

\newcommand{\ek}[1]{{\color{black} #1}}
\newcommand{\dual}[1]{\left\langle#1\right\rangle}
\newcommand{\vnorm}[1]{|\hspace{-0.3mm}|\hspace{-0.3mm}|#1|\hspace{-0.3mm}|\hspace{-0.3mm}|}
\newcommand{\wnorm}[1]{\left|\hspace{-0.3mm}\left|\hspace{-0.3mm}\left|#1\right|\hspace{-0.3mm}\right|\hspace{-0.3mm}\right|}
\definecolor{labelkey}{RGB}{15,82,188}

\begin{document}

\title{Finite element method for radially symmetric solution of a multidimensional semilinear heat equation}
\author{
Toru Nakanishi\thanks{Graduate School of Mathematical Sciences, The University of Tokyo, Komaba 3-8-1, Meguro-ku, Tokyo 153-8914, Japan. \textit{E-mail}: \texttt{nakanish@ms.u-tokyo.ac.jp}, TEL  81-3-5465-7001, FAX  81-3-5465-7011} 
\and 
Norikazu Saito \thanks{Graduate School of Mathematical Sciences, The University of Tokyo, Komaba 3-8-1, Meguro-ku, Tokyo 153-8914, Japan. \textit{E-mail}: \texttt{norikazu@g.ecc.u-tokyo.ac.jp}, TEL  81-3-5465-7001, FAX  81-3-5465-7011}
}

\maketitle

\begin{abstract}
\ek{This study was conducted} to present error \ek{analysis of a} finite element method for computing the radially symmetric solutions of semilinear heat equations.
\ek{Particularly}, this study establishes optimal order error estimates in \ek{$L^\infty$ and weighted
$L^2$ norms, respectively,} for the symmetric and nonsymmetric formulation. Some numerical examples are presented to validate the obtained theoretical results. 
\end{abstract}


\bigskip

{\noindent \textbf{2010 Mathematics Subject Classification:}
65M60,
35K58,

\section{Introduction}

This study \ek{was conducted} to investigate the convergence \tn{property} of finite element method (FEM) applied to a parabolic equation with singular coefficients for the function $u=u(x,t)$, $x\in\overline{I}=[0,1]$\ek{,} and $t\ge 0$, as expressed in
\begin{subequations}
\label{eq:1}
\begin{align}
&u_{t}=u_{xx}+\frac{N-1}{x}u_{x}+f(u), 
 &&x\in I=(0,1),~ t>0, \label{eq:1a}\\
&u_x (0,t)=u(1,t)=0, &&t>0, \label{eq:1b}\\
&u(x,0)=u^0(x),    &&x\in I, \label{eq:1c}
\end{align}
\end{subequations}
where $f$ is a given locally Lipschitz continuous function, $u^0$ is a given continuous function, and 
\begin{equation}
\label{eq:N}
N\ge 2\quad \mbox{integer}
\end{equation}
is a given parameter. 

In the study of 
an $N$\ek{-}dimensional semilinear heat equation, the following problem arises \ek{as}
\begin{subequations}
\label{eq:1z}
\begin{align}
& U_{t}=\Delta U+ f(U), && \bm{x}\in\Omega,~t> 0\\
& U=0, &&\bm{x}\in\partial\Omega,~t>0,\\
& U(0,\bm{x})=U^{0}(\bm{x}), && \bm{x}\in \Omega,
\end{align}
\end{subequations}
where $\Omega$ \ek{represents} a bounded domain \tn{in} $\mathbb{R}^{N}$. 
If one is concerned with the \emph{radially symmetric solution} $u(|\bm{x}|)=U(\bm{x})$ in the $N$\ek{-}dimensional \tn{ball} $\Omega=\{\bm{x}\in\mathbb{R}^N\mid |\bm{x}|=|\bm{x}|_{\mathbb{R}^N} \tn{<} 1\}$, then \eqref{eq:1z} implies \eqref{eq:1}, where $x=|\bm{x}|$ and $u^0 (x)=U_{0}(\bm{x})$. 

For a linear case \ek{in which} $f(u)=0$ is replaced by a given function $f(x,t)$, \tn{the works} \cite{et84,tho06} studied the convergence \tn{property} of the FEM to \eqref{eq:1} along with the \tn{corresponding steady-state problem}, and \ek{two proposed} schemes: the symmetric scheme, wherein they established the optimal order error estimate in the weighted $L^2$ norm\ek{;} and the nonsymmetric scheme, wherein they proved the $L^\infty$ error estimate. \tn{In this paper}, both schemes are applied to the semilinear equation \eqref{eq:1} to derive various error estimates. Moreover, this study includes a discussion of discrete positivity conservation properties, which \ek{earlier} studies \cite{et84,tho06} failed to embrace, but \ek{which} are actually important in the study of diffusion\ek{-}type equations. 

Our \ek{emphasis} is on FEM because we are able to use non-uniform partitions of the space variable\ek{.} \ek{Therefore}, the method \ek{is deemed} useful for examining highly concentrated solutions \tn{at the origin}. On this connection, we present our motivation for this study.  
\ek{The} critical phenomenon appearing in the semilinear heat equation of the form 
\[
 U_{t}=\Delta U+ U^{1+\alpha},\quad \alpha>0
\]
in a multidimensional space has attracted considerable attention since the pioneering work of Fujita \cite{fuj66}. According to him, the equation is in the whole $N$ dimensional space\ek{. Any} positive solution blows up in a finite time if $\alpha\le 2/N$, whereas a solution is smooth at any time for a small initial value if $\alpha>2/N$. 
Therefore, \ek{expression} $p_c=1+2/N$ is known as \ek{Fujita's critical exponent} (\ek{\cite{lev90,dl00} provides some} critical exponents of other equations). Generally, similar critical exponents can be found for an initial-boundary value problem for the semilinear heat equation\ek{. Some examples are given in reports of earlier studies \cite{ish10,lev90,dl00}.} \ek{However}, 
the concrete values of those critical conditions \ek{are apparently unknown}. 
Therefore, we found it interesting to study the numerical methods for computing the solutions of nonlinear partial \ek{differential} equations in an $N$\tn{-}dimensional space. However, computing the non-stationary four-space dimensional problem \ek{is difficult,} even \ek{for} modern computers. 
\ek{We consider} the FEM to solve the one space dimensional equation \eqref{eq:1}. However, we \ek{face} another difficulty in dealing with the singular coefficient $(N-1)/x$, which the FEM reasonably simplified, as \ek{explained} later.\\
 \tn{
 As described above, the main purpose of this paper is to derive various
optimal order error estimates for the symmetric and nonsymmetric
schemes of \cite{et84,tho06} applied to \eqref{eq:1}. These schemes
are described below as (Sym) and (Non-Sym). To this end, we address
mostly the general nonlinearity $f(u)$. Moreover, we study discrete
positivity conservation properties. We summarize our typical results
here.
\begin{itemize}
 \item The solution of (Sym) is positive if $f$ and \ek{if} the discretization
parameters satisfy some conditions, as shown by Theorem \ref{prop:2}.
 \item If $f$ is a \emph{globally} Lipschitz continuous function, then the
solution of (Sym) converges to the solution of \eqref{eq:1} in the
weighted $L^2$ norm for the space and in the $L^\infty$ norm for time.
Moreover, the convergence is at the optimal order, as shown by Theorem
\ref{th:s1}.
 \item If $f$ is a \emph{locally} Lipschitz continuous function and
$N\le 3$, then the solution of (Sym) converges to the solution of
\eqref{eq:1} in the weighted $L^2$ norm for the space and in the
$L^\infty$ norm for time. The convergence is at the optimal order, as shown by Theorem \ref{th:s3}.
 \item If $f(u)=u|u|^\alpha$ with $\alpha\ge 1$ and if the time partition
is uniform,
then the solution of (Non-Sym) converges to the solution of \eqref{eq:1} in
the $L^\infty(0,T;L^\infty(I))$ norm. The convergence is at the optimal
order up to the logarithm factor, as shown by Theorem \ref{th:s6}.
\end{itemize}
However, we do not proceed to applications of our schemes to
the blow-up computation in this work. In fact, from the main
results presented in this paper, we infer that the standard schemes of
\cite{et84,tho06} do not fit for the blow-up computation for large
$N$. For the symmetric scheme, the restriction $N\le 3$ reduces
interest in considering radially symmetric problems.
Moreover, for the nonsymmetric scheme, the use of uniform
time-partitions makes it difficult to apply Nakagawa's time-partitions
control strategy: a powerful technique for computing the
approximate blow-up time, as described in earlier reports \cite{che92,nak76,sai16,cho07,cho10,sai16w,che86}.
Nevertheless, we believe that our results are of interest to
researchers in this and related fields. In fact, the validity issue of the
symmetric scheme only for $N\le 3$ was pointed out earlier in
\cite{akr03} for a nonlinear Schr{\" o}dinger equation with no
mathematical evidence. The analysis reported herein reveals weak
points of the two standard schemes.
As a sequel to this study, we propose a new
finite element scheme for \eqref{eq:1}. The scheme, which uses a nonstandard
mass-lumping approximation, is shown to be positivity-preserving
and convergent for any $N\ge 2$. Details will be reported in a
forthcoming paper.
}\\
 \ek{It is noteworthy that} the finite difference method for \eqref{eq:1} has been studied and \ek{that} its optimal order convergence \ek{was} proved in \ek{an earlier report \cite{che92}.} \ek{Its} finite difference scheme uses a special approximation around the origin to \ek{assume} a uniform spatial mesh.

This paper comprises \tn{five} sections. Section \ref{sec:2} presents our finite element schemes. 
Well-posedness and positivity conservation are examined in Section \ref{sec:3}. 
Section \ref{sec:4} presents the error estimates and their proofs.  
Finally, Section \ref{sec:5} \ek{presents} some numerical examples that validate our theoretical results.

\section{Finite element method}
\label{sec:2}

First, we derive two alternate weak formulations of \eqref{eq:1}. Unless otherwise stated explicitly, we assume that $f$ is a locally Lipschitz continuous function \ek{such} that
\begin{equation}
\tag{f1}
\label{eq:f1}
\forall \mu>0, \ \exists M_\mu>0:\ 
|f(s)-f(s')|\le M_\mu|s-s'|\quad (s,s'\in\mathbb{R},|s|,|s'|\le \mu).
\end{equation}

\ek{Letting} $\chi\in \dot{H}^{1}=\{v\in H^{1}(I)\mid v(1)=0\}$ be arbitrary, then multiplying both sides of \eqref{eq:1a} by $x^{N-1}\chi$ and using integration by parts over $I$, we obtain 
\begin{equation}
 \label{eq:w1}
 \int_I x^{N-1} u_t\chi~dx+
 \int_I x^{N-1} u_{x}\chi_{x}~dx=
 \int_I x^{N-1} f(u)\chi~dx.
\end{equation}
Otherwise, if we multiply both sides of \eqref{eq:1a} by $x\chi$ instead of $x^{N-1}\chi$ and integrate it over $I$, then we have  
\begin{equation}
 \label{eq:w2}
 \int_I x u_t\chi~dx+
 \int_I [x u_{x}\chi_{x}+(2-N)u_{x}\chi]~dx=
 \int_I x f(u)\chi~dx.
\end{equation}
We \ek{designate} \eqref{eq:w1} the \emph{symmetric} weak form \ek{because of} the symmetric bilinear form associated with the differential operator $u_{xx}+\frac{N-1}{x}u_{x}$. 
\ek{In contrast}, \eqref{eq:w2} is the \emph{nonsymmetric} weak form. Both forms are \tn{identical} at $N=2$.


We \ek{now} establish the finite element schemes based on these identities. 
For a positive integer $m$, we introduce node points 
\[
 0=x_0<x_1<\cdots<x_{j-1}<x_{j}<\cdots<x_{m-1}<x_m=1,
\]
and set $I_{j}=(x_{j-1},x_{j})$ and $h_j=x_j-x_{j-1}$, where $j=1,\ldots,m$. The granularity parameter is defined as $h=\max_{1\le j\le m}h_j$. 
Let $\mathcal{P}_k(J)$ be the set of all polynomials in an interval $J$ of degree $\le k$.
We define the $\mathrm{P}1$ finite element space \ek{as} 
\begin{equation}
 \label{eq:2}
S_{h}=\{ v \in H^{1}(I) \mid v\in\mathcal{P}_1(I_j)~(j = 1,\cdots,m),\  v(1)=0\}.
\end{equation}
\ek{Its} standard basis function $\phi_{j}$, $j=0,1,\cdots,m\tn{,}$ is defined as 

\[
 \phi_{j}(x_{i})=\delta_{ij},
\]
where $\delta_{ij}$ denotes \ek{Kronecker's} delta.

For \ek{time} discretization, we \tn{introduce} \ek{non-uniform partitions} 
\[
 t_0=0,\quad t_{n} = \sum_{j=0}^{n-1}\tau_{j} \quad (n\ge 1),
\]
where $\tau_j>0$ denotes the time increments. 

Generally, we write $\partial_{\tau_n}u_h^{n+1}=(u_h^{n+1}-u_h^n)/\tau_n$.\\ 
\tn{
We are now in a position to state the finite element schemes to be considered.
}

\smallskip

\noindent \textbf{(Sym)} Find $u_{h}^{n+1} \in S_{h}$, $n=0,1,\ldots$, such that  

\begin{equation}
\label{eq:3}
\left(\partial_{\tau_n}u_h^{n+1},\chi\right) + A(u_{h}^{n+1},\chi)=(f(u_{h}^{n}),\chi)\quad 
(\chi \in S_{h},~n=0,1,\ldots),
\end{equation} 

where $u_h^0\in S_h$ is assumed to be given. Hereinafter, we set 
\begin{subequations} 
 \label{eq:4}
\begin{align}
 (w,v)&=\int_I x^{N-1} wv~dx, & \|w\|^2=(w,w)=\int_I x^{N-1} w^2~dx, \label{eq:4a}\\
A(w,v)&=\int_I x^{N-1}w_{x}v_{x}~dx. \label{eq:4b}
\end{align} 
\end{subequations}

\medskip


\noindent \textbf{(Non-Sym)} Find $u_{h}^{n+1} \in S_{h}$, $n=0,1,\ldots,$ such that 
\begin{equation}
\label{eq:8}
\dual{\partial_{\tau_n}u_h^{n+1},\chi}+B(u_{h}^{n+1}, \chi)=
\dual{f(u_{h}^{n}),\chi} \quad (\chi\in S_h,~n=0,1,\ldots),
\end{equation}
where 
\begin{subequations} 
 \label{eq:9}
\begin{align}
\dual{w,v}&=\int_Ixwv~dx, \qquad  \vnorm{w}^2=\dual{w,w}=\int_Ixw^2~dx,\label{eq:9a}\\
B(w,v)&=\int_I xw_{x}v_{x}~dx+(2-N)\int_I w_{x}v~dx.\label{eq:9b}
\end{align}
\end{subequations}
 
It is noteworthy that $B(\cdot,\cdot)$ is coercive in $\dot{H}^1$ such that  
\begin{equation}
B(w,w)=\dual{w_{x},w_{x}}+(2-N)\int_I w_{x}wdx=\vnorm{w_{x}}^2+\frac{N-2}{2} w(0)^{2}\ge\vnorm{w_{x}}^2. \label{eq:bc2}
\end{equation}

\section{Well-posedness and positivity conservation}
\label{sec:3}

In this section, we \ek{prove} the following theorems. 

\begin{thm}[Well-posedness of \textup{(Sym)}]
\label{prop:1}
For a given $u_h^n\in S_h$ with $n\ge 0$, the scheme \textup{(Sym)} admits a unique solution $u_h^{n+1}\in S_h$. 
\end{thm}

\begin{thm}[Positivity of (Sym)]
\label{prop:2}
In addition to the basic assumption \eqref{eq:f1}, assume that 
\begin{equation}
 \tag{f2}
\label{eq:f2}
\mbox{$f$ is a non-decreasing function with $f(0)\ge 0$}.
\end{equation}
\ek{Letting} $n\ge 0$ and $u_{h}^n\ge 0$, and \ek{assuming} that 
\begin{equation}
 \label{eq:tau1}
\tau_n\ge \frac{1}{4}h^2\ek{,}
\end{equation}   
\ek{then}, the solution $u_h^{n+1}$ of \textup{(Sym)} satisfies $u_h^{n+1}\ge 0$. 
\end{thm}

\begin{thm}[Comparison principle for (Sym)]
\label{prop:3}
\ek{We let} $n\ge 0$ and assume that $u_{h}^{n},\tilde{u}_h^n\in S_{h}$ satisfies $u_h^n\le \tilde{u}_{h}^{n}$ in $I$. Furthermore, we assume that \eqref{eq:f1} and \eqref{eq:f2} are satisfied. 
\ek{Similarly, we let} $u_{h}^{n+1},\tilde{u}_h^{n+1}\in S_{h}$ be the solutions of \textup{(Sym)} with $u_{h}^{n},\tilde{u}_h^n$, respectively, using the same time increment $\tau_n$. 
Moreover, \ek{we} assume that \eqref{eq:tau1} is satisfied. \ek{Consequently}, we obtain $u_{h}^{n+1}\le \tilde{u}_{h}^{n+1}$ in $I$\ek{. The equality} holds true if and only if $u_h^n=\tilde{u}_h^n$ in $I$.   
\end{thm}

\begin{thm}[Well-posedness of \textup{(Non-Sym)}]
\label{prop:11}
For a given $u_h^n\in S_h$ with $n\ge 0$, the scheme \textup{(Non-Sym)} admits a unique solution $u_h^{n+1}\in S_h$.  
\end{thm}

To prove these theorems, we conveniently rewrite \eqref{eq:3} into a matrix form.  
That is, we introduce
\begin{align*}
&\mathcal{M}=(\mu_{i,j})_{0\le i,j\le m-1}\in\mathbb{R}^{m\times m}, 
&& \mu_{i,j}=(\phi_{j},\phi_{i}),\\
&\mathcal{A}=(a_{i,j})_{0\le i,j\le m-1}\in\mathbb{R}^{m\times m} ,
&& a_{i,j}=A(\phi_{j},\phi_{i}),\\
&\bm{u}^{n}=(u_{j}^{n})_{0\le j\le m-1}\in\mathbb{R}^{m}, && u_j^n=u_{h}^{n}(x_{j}),\\
&\bm{F}^{n}=(F_{j}^{n})_{0\le j\le m-1}\in\mathbb{R}^{m}, && F_j^n=(f(u_{h}^{n}),\phi_{j}),
\end{align*}
and express \eqref{eq:3} as 
\begin{equation}
 \label{eq:3m}
(\mathcal{M}+\tau_{n}\mathcal{A})\bm{u}^{n+1}=\mathcal{M}\bm{u}^{n}+\tau_{n}\bm{F}^{n}\quad (n=0,1,\ldots),
\end{equation}
where $u_m^n=u_h^n(x_m)$ is understood as $u_m^n=0$.

\begin{lemma}
\label{la:1}
$\mathcal{M}$ and $\mathcal{A}$ are both tri-diagonal and positive-definite matrices.  
\end{lemma}

\tn{Theorem \ref{prop:1} is a direct consequence of this lemma. We proceed to proofs of other theorems.}

\begin{proof}[Proof of Theorem \ref{prop:2}]
We use the representative matrix \eqref{eq:3m} instead of \eqref{eq:3} and set 
\[
 \mathcal{C}=(c_{i,j})_{0\le i,j\le m-1}=\mathcal{M}+\tau_{n}\mathcal{A},\quad 
c_{i,j}=\mu_{i,j}+\tau_n a_{i,j}.
\] 
If $\mathcal{C}^{-1}\ge O$, then we \ek{obtain} 
\[
 \bm{u}^{n+1}=\mathcal{C}^{-1}\left(\mathcal{M}\bm{u}^{n}+\tau_{n}\bm{F}^{n}\right)\ge \bm{0},
\]
\ek{because} $\mathcal{M}\ge O$ and $\bm{F}^n\ge \bm{0}$ in view of (f2).  
The proof that $\mathcal{C}^{-1}\ge O$ is true under \eqref{eq:tau1} is divided into three steps, each described as \ek{presented} below. 

\noindent \emph{Step 1.} We show that 
\begin{equation}
 \label{eq:p21}
\sum_{j=0}^{m-1}c_{i,j}>0\qquad (0\le i\le m-1).
\end{equation}   
Letting $1\le i\le m-2$, we calculate 
\begin{align*}
 \sum_{j=0}^{m-1}c_{i,j} 
&= \sum_{j=i-1}^{i+1}\mu_{i,j}+\tau_n\sum_{j=i-1}^{i+1} a_{i,j}\\
&= \sum_{j=i-1}^{i+1}\mu_{i,j}+\tau_n \int_{x_{i-1}}^{x_{i+1}}x^{N-1}(\phi_{i-1}+\phi_i+\phi_{i+1})_{x}(\phi_i)_{x}~dx\\
&= \sum_{j=i-1}^{i+1}\mu_{i,j}>0,
\end{align*}
\ek{because} $\phi_{i-1}+\phi_i+\phi_{i+1}\equiv 1$ in $(x_{i-1},x_{i+1})$. \ek{Cases} $i=0$ and $i=m-1$ are verified similarly. 

\noindent \emph{Step 2.} We show that, if  
\begin{equation}
 \label{eq:p25}
\tau_{n}\ge-\frac{\mu_{i,i+1}}{a_{i,i+1}},-\frac{\mu_{i,i-1}}{a_{i,i-1}}\qquad (i=0,1,\cdots,m-1),
\end{equation}
then $\mathcal{C}^{-1}\ge O$. First, \eqref{eq:p25} implies that $c_{i,i-1},c_{i,i+1}\le 0$ for $0\le i\le m-1$ \ek{because} $a_{i,i-1},a_{i,i+1}<0$.  
Matrix $\mathcal{C}$ is decomposed as $\mathcal{C}=\mathcal{D}(\mathcal{I}-\mathcal{E})$, where $\mathcal{D}=(d_{i,j})_{0\le i,j\le m-1}$ and $\mathcal{E}=(e_{i,j})_{0\le i,j\le m-1}$ are defined as
\[
d_{i,j}=
\begin{cases}
c_{i,i} & (i=j)\\
0 & (i\neq j)
\end{cases}
,\qquad 
e_{i,j}=
\begin{cases}
0& (i=j)\\
-\frac{c_{i,j}}{c_{i,i}} & (i\neq j),
\end{cases}
\]
and \ek{where} $I$ is the identity matrix.
Apparently, $\mathcal{I}-\mathcal{E}$ is non-singular and $\mathcal{D}\ge O$. 
Using \eqref{eq:p21}, we deduce 
\[
 \|\mathcal{E}\|_\infty
=\max_{0\le i\le m-1}\left( -\frac{c_{i,i-1}}{c_{i,i}}-\frac{c_{i,i+1}}{c_{i,i}}\right)<1.
\]
Therefore, matrix $\mathcal{I}-\mathcal{E}$ is non-singular and $(\mathcal{I}-\mathcal{E})^{-1}=\sum_{k=0}^\infty\mathcal{E}^k\ge O$. Consequently, we have $\mathcal{C}^{-1}=(\mathcal{I}-\mathcal{E})^{-1}\mathcal{D}^{-1}\ge O$.  

\noindent \emph{Step 3.} Finally, we \ek{demonstrate} that \eqref{eq:tau1} implies \eqref{eq:p25}. We calculate   
\begin{align*}
\mu_{i,i+1}& = \int_{x_{i}}^{x_{i+1}}x^{N-1}\frac{1}{h_{i+1}^{2}}(x-x_{i})(x_{i+1}-x)~dx
          \le\frac{1}{4}h_{i+1}^{2}\int_{x_{i}}^{x_{i+1}}\frac{1}{h_{i+1}^{2}}x^{N-1}~dx,\\
-a_{i,i+1}&=\int_{x_{i}}^{x_{i+1}}x^{N-1}\frac{1}{h_{i+1}^{2}}~dx.          
\end{align*}
Therefore, we deduce $-\frac{\mu_{i,i+1}}{a_{i,i+1}}\le \frac{1}{4}h^{2}$. 
\end{proof}

\begin{proof}[Proof of Theorem \ref{prop:3}]
\ek{Because} $f(\tilde{u}_h^n)-f(u_h^{n})\ge 0$ in $I$, the proof follows exactly the same \ek{pattern} as \ek{that of} the proof of Proposition \ref{prop:2}.     
\end{proof}

\ek{We} proceed to the result for \textup{(Non-Sym)}: 

\begin{align*}
&\mathcal{M}'=(\mu_{i,j}')_{0\le i,j\le m-1}\in\mathbb{R}^{m\times m}, 
&& \mu_{i,j}'=\dual{\phi_{j},\phi_{i}},\\
&\mathcal{B}=(b_{i,j})_{0\le i,j\le m-1}\in\mathbb{R}^{m\times m} ,
&& b_{i,j}=B(\phi_{j},\phi_{i}),\\
&\bm{G}^{n}=(G_{j}^{n})_{0\le j\le m-1}\in\mathbb{R}^{m}, && G_j^n=\dual{f(u_{h}^{n}),\phi_{j}},
\end{align*}
and express \eqref{eq:8} as 
\begin{equation}
 \label{eq:8m}
(\mathcal{M}'+\tau_{n}\mathcal{B})\bm{u}^{n+1}=\mathcal{M}'\bm{u}^{n}+\tau_{n}\bm{G}^{n}\quad (n=0,1,\ldots).
\end{equation}

\tn{
In view of \eqref{eq:bc2}, $\mathcal{M}'$ and $\mathcal{B}$ are both tri-diagonal and positive-definite matrices. 
Therefore, the proof is completed.
}

\section{Convergence and error analysis}
\label{sec:4}

\subsection{Results}
\label{sec:41}

Our convergence results for (Sym) and (Non-Sym) are stated under a smoothness assumption \ek{of} the solution $u$ of \eqref{eq:1}\ek{: given} $T>0$ and setting $Q_T=[0,1]\times [0,T]$, we assume that $u$ is sufficiently smooth such that 
\begin{equation}
\label{eq:smooth1}
\kappa_\nu(u)=
\sum_{j=0}^2 \|\partial_x^j u\|_{L^{\infty}(Q_T)}
+\sum_{l=1}^{2+\nu} \|\partial_t^l u\|_{L^{\infty}(Q_T)}
+\sum_{k=1}^{1+\nu}\|\partial_t^k\partial_x^2u\|_{L^{\infty}(Q_T)}
<\infty,
\end{equation}
where $\nu$ is either 0 or 1. 

The partition $\{x_i\}_{j=0}^m$ of $\bar{I}=[0,1]$ is assumed to be quasi\ek{-}uniform, \ek{with} a positive constant $\beta$ independent of $h$ such that 
\begin{equation}
\label{eq:beta}
h \le \beta \min_{1\le j \le m}h_{j}.
\end{equation}
Finally, the approximate initial value $u_h^0$ is chosen as 
\begin{equation}
\label{eq:iv1}
 \|u_h^0-u^0\|\le C_0h^2
\end{equation}
for a positive constant $C_0$. 

Moreover, for $k=1,2,\ldots$, we express the positive \ek{constants} $C_k=C_k(\gamma_1,\gamma_2,\ldots)$ and $h_k=h_k(\gamma_1,\gamma_2,\ldots)$ according to the parameters $\gamma_1,\gamma_2,\ldots$. \ek{Particularly}, $C_k$ and $h_k$ are independent of $h$ and $\tau$. 

\ek{Next} we state the following theorems. 

\begin{thm}[Convergence for (Sym) in $\|\cdot\|$, I]
\label{th:s1} 
Assume that $f$ is a globally Lipschitz continuous function; assume \eqref{eq:f1} and   
\begin{equation}
\tag{f3}
 \label{eq:f3}
M=\sup_{\mu>0}M_\mu<\infty.
\end{equation}
Assume that, for $T>0$, solution $u$ of \eqref{eq:1} is sufficiently smooth \ek{that} \eqref{eq:smooth1} for $\nu=0$ holds true. 
Moreover, assume that \eqref{eq:beta} and \eqref{eq:iv1} are satisfied. 
Then, there \ek{exists} $h_1=h_1(N,\beta)$ such that, for any $h\le h_1$, we have
\[
\sup_{0\le t_n\le T}\|u_{h}^{n}-u(\cdot,t_{n})\|\le C_1(h^{2}+\tau),
\]
where $C_1=C_1(T, M, \kappa_0(u), C_0, N,\beta)$ and $u_h^n$ is the solution of \textup{(Sym)}.  
\end{thm}

\ek{For} $L^\infty$ error estimates, we \ek{must} further assume that $u_h^0$ is chosen as 
\begin{equation}
\label{eq:iv2}
 A(u_h^0-u^0,v_h)=0 \quad (v_h\in S_h).
\end{equation}

\begin{thm}[Convergence for (Sym) in $\|\cdot\|_{L^\infty(\sigma,1)}$, I]
\label{th:s2}
In addition to the assumption of Theorem \ref{th:s1}, assume that \eqref{eq:iv2} is satisfied. Furthermore, let $\sigma\in(0,1)$ be arbitrary. 
Then, there exists an $h_2=h_2(N,\beta)$ such that, for any $h\le h_2$, we have 
\[
\sup_{0\le t_n\le T}\|u_{h}^{n}-u(\cdot,t_{n})\|_{L^{\infty}(\sigma,1)}\le C_2\left(h^{2}\log\frac{1}{h}+\tau\right),
\]                                                                                        
where $C_2=C_2(T, M, \kappa_0(u), C_0, N,\beta,\sigma)$ and $u_h^n$ is the solution of \textup{(Sym)}.                                                                        
\end{thm}  

The restriction that $f$ is a globally Lipschitz continuous function with (f3) can be removed in \ek{the following manner.}  

\begin{thm}[Convergence of (Sym) in $\|\cdot\|$, II]
\label{th:s3}
Given \ek{that} $T>0$ and \ek{that} only \eqref{eq:f1} is satisfied, 
we assume that \eqref{eq:smooth1} with $\nu=0$, \eqref{eq:beta}, 
and \eqref{eq:iv1} are satisfied.  
Furthermore, assume that $N\le 3$ and 
that there exist positive constants $c_1$ and $\sigma$ such that 
\begin{equation}
 \label{eq:mesh}
\tau h^{-N/2}\le c_1h^{\sigma}.  
\end{equation}
\ek{Then} there exists an $h_3=h_3(T, \kappa_0(u), C_0, N,\beta)$ such that, for any $h\le h_3$, we have 
\[
\sup_{0\le t_n\le T}\|u_{h}^{n}-u(\cdot,t_{n})\|\le C_2(h^{2}+\tau),
\]
where $C_3=C_3(T, \kappa_0(u), C_0, N,\beta)$ and $u_h^n$ is the solution of \textup{(Sym)}.          \end{thm}

\begin{thm}[Convergence for (Sym) in $\|\cdot\|_{L^\infty(\sigma,1)}$, II]
\label{th:s4}
Given \ek{that} $T>0$ and \ek{that} \eqref{eq:f1} is satisfied, 
we assume that \eqref{eq:smooth1} with $\nu=0$, \eqref{eq:beta}, 
\eqref{eq:iv1}, \eqref{eq:iv2} and \eqref{eq:mesh} are satisfied.  
\ek{Consequently}, there \ek{exists} $h_4=h_4(T, \kappa_0(u), C_0, N,\beta)$ such that, for any $h\le h_4$, we have  
\[
\sup_{0\le t_n\le T}\|u_{h}^{n}-u(\cdot,t_{n})\|_{L^{\infty}(\sigma,1)}
\le C_4\left(h^{2}\log\frac{1}{h}+\tau\right),
\]
where $C_4=C_4(T, \kappa_0(u), C_0, N,\beta)$ and $u_h^n$ is the solution of \textup{(Sym)}.                                            
\end{thm}           

\tn{Subsequently}, let us proceed to \ek{error} estimates for (Non-Sym). 
For the approximate initial value $u_h^0$, we choose        
\begin{equation}
\label{eq:iv3}
 B(u_h^0-u^0,v_h)=0\qquad (v_h\in S_h).
\end{equation}
Quasi-uniformity is \tn{also} required for the time partition\ek{. Therefore,} 
there exists a positive constant $\gamma>0$ such that 
\begin{equation}
\label{eq:qt}
\tau \le \gamma \tau_{\min}\ek{,}
\end{equation}
where $\tau_{\min}=\min_{n\ge 0}\tau_n$. Moreover, we set 
\begin{equation}
 \label{eq:tau}
\delta=\sup_{\tn{t_{k+1}}\in [0,T]}|\tau_{k}-\tau_{k+1}|. 
\end{equation}

\begin{thm}[Convergence for (Non-Sym), I]
\label{th:s5}
Let $f$ be a $C^1$ function satisfying 
\begin{equation}
\tag{f4}
\label{eq:f4} 
M_1=\sup_{s\in\mathbb{R}}|f'(s)|<\infty,\quad 
M_2=\sup_{s\ne s'\in \mathbb{R}}\frac{|f'(s)-f'(s')|}{|s-s'|}<\infty.
\end{equation}
Given $T>0$, 
we assume that the solution $u$ of \eqref{eq:1} is sufficiently smooth \ek{that} \eqref{eq:smooth1} for $\nu=1$ holds true. Furthermore, we assume that \eqref{eq:beta}, \eqref{eq:iv3} and \eqref{eq:qt} are satisfied. Then, there exists an $h_5=h_5(T, \kappa_1(u), M_1,M_2,\gamma,N,\beta)$ such that, for any $h\le h_5$, we have  
 \[
\sup_{0\le t_n\le T}\|u_{h}^{n}-u(\cdot,t_{n})\|_{L^{\infty}(I)} 
\le C_5\left(\log\frac{1}{h}\right)^{\frac{1}{2}}\left(h^{2}+\tau+\frac{\delta}{\tau_{\min}}\right) ,
\]
where $C_5=C_5(T, \kappa_1(u), M_1,M_2,\gamma,N,\beta)>0$ and $u_h^n$ is the solution of \textup{(Non-Sym)}. 
\end{thm}                                                                                                                
Finally, we state the error estimates for non-globally Lipschitz continuous function $f$. To avoid \ek{unnecessary} complexity, we deal \ek{only} with the power nonlinearity $f(s)=s|s|^\alpha$.  

\begin{thm}[Convergence for (Non-Sym), II]
\label{th:s6}
\ek{Letting} $f(s)=s|s|^\alpha$for $s\in\mathbb{R}$, where $\alpha\ge 1$\ek{, then given} $T>0$, 
we assume that \eqref{eq:smooth1} with $\nu=1$, \eqref{eq:beta} and \eqref{eq:iv3} are satisfied.  
Then, there exists an $h_6=h_6(T,\kappa_1(u),\gamma,N,\beta)$ such that, for any $h\le h_6$, we have  
\[
\sup_{0\le t_n\le T}\|u_{h}^{n}-u(\cdot,t_{n})\|_{L^{\infty}(I)}
\le C_6\left(\log\frac{1}{h}\right)^{\frac{1}{2}}(h^{2}+\tau),
\]
where $C_6=C_6(T, \kappa_1(u),\gamma,N,\beta)$ and $u_h^n$ is the solution of \textup{(Non-Sym)}. 
\end{thm}

\subsection{Proof of Theorems \ref{th:s1} and \ref{th:s2}}
\label{sec:43}

We use the projection operator $P_A$ of $\dot{H}^1\to S_h$ associated with $A(\cdot,\cdot)$, defined for $w\in \dot{H}^1$ as 
\begin{equation}
P_Aw\in S_h,\quad A(P_{A}w -w,\chi)=0\qquad (\chi\in S_{h}) .
 \label{eq:pA}
\end{equation}

In \cite{et84} and \cite{jes78}, the \ek{following} error estimates \ek{are proved}.  

\begin{lemma}
\label{prop:tj}
Letting $w\in C^{2}(\bar{I})\cap\dot{H}^1$, and \eqref{eq:beta} be satisfied, \ek{then} for $h\le h_7=h_7(N,\beta)$, we obtain 
\begin{align}
\|P_Aw-w\| &\le Ch^{2}\|w_{xx}\|,  \label{eq:tj1} \\
\|P_Aw-w\|_{L^{\infty}(I)}&\le C\left( \log \frac{1}{h} \right)h^{2}\|w_{xx}\|_{L^{\infty}(I)}, \label{eq:tj2}
\end{align}
where $C$ is a positive constant depending only on $N$ and $\beta$. 
\end{lemma}


\begin{proof}[Proof of Theorem \ref{th:s1}]
Using $P_Au$, we distribute the error in the form \ek{shown below.} 
\[
u_{h}^{n}-u(t_{n})=\underbrace{(u_{h}^{n}- P_{A}u(t_{n}))}_{=\theta^{n}} + \underbrace{(P_{A}u(t_{n}) - u(t_{n}))}_{=\rho^{n}}\\ 
\]
 \ek{From \eqref{eq:tj1}, it is known} that
\begin{equation}
\label{eq:th1.12}
\|\rho^{n}\| \le Ch^{2}\|u_{xx}(t_n)\| \le Ch^2 \|u_{xx}\|_{L^\infty(Q_T)}.
\end{equation}
\ek{Next} we derive \ek{an estimate} for $\theta^n$.  
By considering the symmetric weak form \eqref{eq:w1} at $t=t_{n+1}$, we \ek{obtain} 
\begin{multline*}
\left(\partial_{\tau_n} u(t_{n+1}),\chi\right) +A(P_{A}u(t_{n+1}),\chi)= (f(u(t_{n})),\chi)
                                                      \\
+( f(u(t_{n+1}))-f(u(t_{n})),\chi)           
+\left( \partial_{\tau_n}u(t_{n+1})-u_{t}(t_{n+1}), \chi\right)
\end{multline*}
which, together with \eqref{eq:3}, implies that
\begin{multline}
\label{eq:th1.10}
\left( \partial_{\tau_n}\theta^{n+1},\chi\right)+A(\theta^{n+1},\chi)
=(f(u_h^n)-f(u(t_{n})),\chi) \\
 - (f(u(t_{n+1}))-f(u(t_{n})),\chi)
 - \left(\partial_{\tau_n}u(t_{n+1})-u_{t}(t_{n+1}),\chi \right)-\left(\partial_{\tau_n}\rho^{n+1},\chi \right).
\end{multline}
Substituting this \ek{expression} for $\chi=\theta^{n+1}$ \ek{yields the following:}
\begin{multline*}
\frac{1}{\tau_{n}}\left\{ \|\theta^{n+1}\|^{2}-\|\theta^{n}\|
\cdot\|\theta^{n+1}\|\right\} 
\le M\|\theta^{n}+\rho^{n}\|\cdot\|\theta^{n+1}\| \\
+ M\tau_{n}\|u_{t}\|_{L^{\infty}(Q_T)}\cdot\|\theta^{n+1}\|
+C\tau_{n} \|u_{tt}\|_{L^{\infty}(Q_T)}\|\theta^{n+1}\|
+\left\| \partial_{\tau_n}\rho^{n+1}\right\|\cdot\|\theta^{n+1}\|.
\end{multline*}
Correspondingly, \ek{because} 
\[
 \partial_{\tau_n}\rho^{n+1}
=P_{A}\left(\frac{u(t_{n+1})-u(t_{n})}{\tau_{n}}\right) -\frac{u(t_{n+1})-u(t_{n})}{\tau_{n}},
\]
we provide an estimate 
\begin{equation}
 \label{eq:th1.11}
\left\| \partial_{\tau_n}\rho^{n+1}\right\|
\le Ch^{2}\left\|\frac{u_{xx}(t_{n+1})-u_{xx}(t_{n})}{\tau_{n}}\right\|\\
\le Ch^{2}\|u_{xxt}\|_{L^{\infty}(Q_T)}.
\end{equation}
To sum up, we obtain  
\[
\|\theta^{n+1}\|-\|\theta^{n}\| \le 
\tau_{n}M\|\theta^{n}\| \\
+ Ch^{2}M\tau_{n}+ CM\tau_{n}^{2}
+C\tau_{n}^{2} +Ch^{2}\tau_{n}.
\]
\ek{Therefore,} 
\begin{align}
 \|\theta^{n}\| 
&\le  e^{MT}\|u_{h}^{0}-P_{A}u^{0}\|+C\frac{e^{MT}-1}{M}(\tau +h^{2}) \nonumber \\
&\le  e^{MT}(\|u_{h}^{0}-u^{0}\|+\|u^{0}-P_{A}u^{0}\|)+C\frac{e^{MT}-1}{M}(\tau +h^{2}) \nonumber \\ 
&\le  C'(\tau +h^{2}), 
\label{eq:th1.14}
\end{align}
where $C'=C'(T,\kappa_0(u),M,N,\beta,C_{0})>0$. 
\ek{By combining} this expression with \eqref{eq:th1.12}, \ek{one can} deduce the desired error estimate. 
\end{proof}

\begin{proof}[Proof of Theorem \ref{th:s2}]
We use the same error decomposition process as \ek{that used} in the previous proof where 
$u_{h}^{n}-u(t_{n})=\theta^{n}+\rho^{n}$\ek{. Also, we} apply \eqref{eq:tj2} to estimate $\|\rho^{n}\|_{L^{\infty}(I)}$. \ek{Because}  
\begin{equation}
 \label{eq:th2.23}
 \|\theta^{n}\|_{L^{\infty}(\sigma,1)}\le\|\theta^{n}_{x}\|_{L^{1}(\sigma,1)}
\le C(\sigma,N)\|\theta^{n}_{x}\|,
\end{equation}
we perform an estimation for $\|\theta^{n}_{x}\|$. 

Substituting \eqref{eq:th1.10} for $\chi=\partial_{\tau_n}\theta^{n+1}$, we \ek{obtain the following.}                                                     
\begin{multline*}                                                        
\left\|\partial_{\tau_n}\theta^{n+1}\right\|^{2}
+A(\theta^{n+1},\partial_{\tau_n}\theta^{n+1})
\le  M\|\theta^{n}\|\cdot\left\|\partial_{\tau_n}\theta^{n+1}\right\|\\
+M\|\rho^{n}\|\cdot \left\|\partial_{\tau_n}\theta^{n+1}\right\|
+M\tau_{n}\|u_{t}\|_{L^{\infty}(Q_T)}\cdot\left\|\partial_{\tau_n}\theta^{n+1} \right\|  \\
+\|u_{tt}\|_{L^{\infty}(Q_T)}\tau_{n}\left\|\partial_{\tau_n}\theta^{n+1}\right\|+\left\|\partial_{\tau_n}\rho^{n+1}\right\|\cdot\left\|\partial_{\tau_n}\theta^{n+1}\right\|                                                         
\end{multline*}
                                                               
Correspondingly, we apply \ek{the} elementary identity \ek{shown below} 
\begin{align*}
A\left(\theta^{n+1},\partial_{\tau_n}\theta^{n+1}\right) 
&=  
\frac12 A\left(\theta^{n+1}-\theta^n+\theta^{n+1}+\theta^n,\partial_{\tau_n}\theta^{n+1}\right)  \\
&\ge 
\frac{1}{2\tau_n}\left[ 
A\left(\theta^{n+1}, \theta^{n+1}\right)-
A\left(\theta^{n}, \theta^{n}\right) \right]
\end{align*}
along with Young's inequality \ek{to} obtain  
\begin{multline*}
\frac{1}{2\tau_{n}}\left[A(\theta^{n+1},\theta^{n+1})- A(\theta^{n},\theta^{n})\right] 
\le
\frac{1}{2}\frac{M^{2}}{\delta_{0}^{2}}\|\theta^{n}\|^{2}
+\frac{1}{2}\delta_{0}^{2}\left\|\partial_{\tau_n}\theta^{n+1}\right\|^{2} \\
+\frac{1}{2}\frac{M^{2}}{\delta_{1}^{2}}\|\rho^{n}\|^{2} 
+\frac{1}{2}\delta_{1}^{2}\left\|\partial_{\tau_n}\theta^{n+1}\right\|^{2} 
+\frac{1}{2}\frac{C^{2}}{\delta_{2}^{2}}\tau_{n}^{2}
+\frac{1}{2}\delta_{2}^{2}\left\|\partial_{\tau_n}\theta^{n+1}\right\|^{2} \\
+ \frac{1}{2}\left\|\partial_{\tau_n}\rho^{n+1}\right\|^{2}                               +\frac{1}{2}\left\|\partial_{\tau_n}\theta^{n+1}\right\|^{2}
-\left\|\partial_{\tau_n}\theta^{n+1} \right\|^{2}, 
\end{multline*}
where $\delta_0,\delta_1,\delta_2>0$ are constants.     
After setting $\delta_{0}^{2}+\delta_{1}^{2}+\delta_{2}^{2}=1$, we \ek{obtain} 
\[
A(\theta^{n+1},\theta^{n+1})-A(\theta^{n},\theta^{n})\le 
\tau_{n}\left[\frac{C^{2}}{\delta_{0}^{2}}\|\theta^{n}\|^{2}+\frac{C^{2}}{\delta_{1}^{2}}\|\rho^{n}\|^{2}+\left\|\partial_{\tau_n}\rho^{n+1}\right\|^{2}+\frac{C^{2}}{\delta_{2}^{2}}\tau^2\right].
\]
Therefore, 
\[
A(\theta^{n},\theta^{n})\le A(\theta^0,\theta^0)+C^2t_{n} \sup_{1\le k\le n}\left[\|\theta^{k-1}\|^{2}+\|\rho^{k-1}\|^{2}+\left\|\partial_{\tau_{k-1}}\rho^{k} \right\|^{2}+\tau^2\right] .
\]
Consequently, using \eqref{eq:iv2}, \eqref{eq:th1.11}, and \eqref{eq:th1.14}, we deduce 
\[
\|\theta^{n}_{x}\| \le C t_{n}^{\frac{1}{2}}\left(\tau+h^{2}\right).
\]
This, together with \eqref{eq:tj2} and \eqref{eq:th2.23}, implies the desired estimate. 
\end{proof}

\subsection{Proof of Theorems \ref{th:s3} and \ref{th:s4} }
\label{sec:45}

For the proof, we \ek{use} the inverse inequality that follows.

\begin{lemma}[Inverse inequality]
\label{la:ie}
Under condition \eqref{eq:beta},  
\[
\|v_{h}\|_{L^{\infty}(I)}\le C_{\star}h^{-\frac{N}{2}}\|v_{h}\| \qquad (v_h\in S_h),
\]
where $C_{\star}$ is a positive constant depending only on $N$ and $\beta$. 
\end{lemma}

\begin{proof}
Let $v_h\in S_h$ be arbitrary. 
 From the norm equivalence in $\mathbb{R}^2$, we know that 
\begin{align*}
 \|v_h\|_{L^\infty(I_1)} &\le C_{\star\star}h_1^{-1/2}\|v_h\|_{L^2(\frac{h_1}{2},h_1)},\\
 \|v_h\|_{L^\infty(I_j)} &\le C_{\star\star}h_j^{-1/2}\|v_h\|_{L^2(I_j)}\quad (j=2,\ldots,m),
\end{align*}
where $C_{\star\star}$ denotes the absolute positive constant. Given that $\|v_h\|_{L^\infty(I)}=\|v_h\|_{L^\infty(I_1)}$, the expression \ek{is calculable} as 
\begin{align*}
 \|v_h\|_{L^\infty(I_1)}^2
&\le C_{\star\star}^2h_1^{-1}\int_{h_1/2}^{h_1} x^{-(N-1)}x^{N-1}v_h^2~dx\\
&\le C_{\star\star}^2h_1^{-1}\left(\frac{h_1}{2}\right)^{-(N-1)}\int_{h_1/2}^{h_1} x^{N-1}v_h^2~dx\\
&\le C_{\star\star}^22^{N-1}h^{-N}\left(\frac{h_1}{h}\right)^{-N}\int_{h_1/2}^{h_1} x^{N-1}v_h^2~dx\\
&\le C_{\star}^2 h^{-N} \|v_h\|^2.
\end{align*}
The case $\|v_h\|_{L^\infty(I)}=\|v_h\|_{L^\infty(I_j)}$ with $j=2,\ldots,m$ is examined similarly. 
\end{proof}

\begin{proof}[Proof of Theorem \ref{th:s3}]
Consider \eqref{eq:1} and (Sym) with replacement $f(s)$ in
\[
\tilde{f}(s)=
\begin{cases}
f(\mu) & (s\ge \mu)\\
f(s)& (-\mu\le s\le\mu)\\
f(-\mu) & (s\le -\mu),
\end{cases}
\]
where $\mu>0$ is determined later. Then, $\tilde{f}$ satisfies condition \eqref{eq:f3} in Theorem \ref{th:s1} such that 
\[
\sup_{s,s'\in\mathbb{R},s\ne s'}\frac{|\tilde{f}(s)-\tilde{f}(s')|}{|s-s'|}\le M\equiv \sup_{|\lambda|\le \mu}M_\lambda<\infty.
\]              
Let $\tilde{u}$ and $\tilde{u}_{h}^{n}$ be the solutions of \eqref{eq:1} and (Sym) with $\tilde{f}$, respectively, such that 

\[
\|\tilde{u}_{h}^{n}\|_{L^{\infty}(I)}\le\|\theta^{n}\|_{L^{\infty}(I)}+\|P_{A}\tilde{u}(t_{n})\|_{L^{\infty}(I)},
\]            
where $\theta^{n}=\tilde{u}_{h}^{n}-P_{A}\tilde{u}(t_{n})$ and $\rho^{n}= P_{A}\tilde{u}(t_{n})-\tilde{u}(t_{n})$.                
Applying Theorem \ref{th:s1} to $\tilde{u}$ and $\tilde{u}_{h}^{n}$, 
\ek{one obtains} 
\begin{equation}
\label{eq:s2.1}
\sup_{0\le t_n\le T}\|\tilde{u}_{h}^{n}-\tilde{u}(\cdot,t_{n})\|\le C_2(h^{2}+\tau),
\end{equation}
where $C_2=C_2(T,\kappa_0(\tilde{u}),\mu,C_0, N,\beta)$.  
Moreover, \ek{an} estimate \eqref{eq:th1.14} for $\theta^n$ is available. In view of Lemmas \ref{prop:tj} and \ref{la:ie}, we determine \ek{those estimates as} 
\begin{align*}
\|\theta^{n}\|_{L^{\infty}(I)}
& \le C_\star h^{-\frac{N}{2}}\|\theta^{n}\| \le C_3h^{-\frac{N}{2}}(h^{2}+\tau),\\
\|\rho\|_{L^{\infty}(I)}&\le C_4\left(h^{2}\log\frac{1}{h}\right)\|\tilde{u}_{xx}(t_{n})\|_{L^\infty(I)},
\end{align*}
where $C_3=C_3(T,\kappa_0(\tilde{u}),\mu,C_0,N,\beta)$ and $C_4=C_4(N,\beta)$. 
Therefore, we have 
\[
\|P_{A}\tilde{u}(t_{n})\|_{L^{\infty}(I)}\le\|\tilde{u}(t_{n})\|_{L^{\infty}(I)}+C_5\left(h^{2}\log\frac{1}{h}\right)\|\tilde{u}_{xx}(t_{n})\|_{L^{\infty}(I)}
\]
and 
\[
\|\tilde{u}_{h}^{n}\|_{L^{\infty}(I)}\le C_3(h^{2-\frac{N}{2}} +h^{-\frac{N}{2}}\tau)+\|\tilde{u}(t_{n})\|_{L^{\infty}(I)}+C_4\left(h^{2}\log\frac{1}{h}\right)\|\tilde{u}_{xx}(t_{n})\|_{L^{\infty}(I)}.
\]

At this stage, we set $\mu=1+\|u\|_{L^{\infty}(Q_T)}$ to obtain $u=\tilde{u}$ in $Q_T$ by uniqueness. Moreover, \ek{because} $N<4$, we can take a very small $h$ such that 
\[
C_{6}(h^{2-\frac{N}{2}}+h^{-\frac{N}{2}}\tau)\le \frac{1}{2},\quad 
C_{5}\left(h^{2}\log\frac{1}{h}\right)\|u_{xx}(t_{n})\|_{L^{\infty}(I)}\le \frac{1}{2}.                           
\]
Consequently, $\|\tilde{u}_{h}^{n}\|_{L^{\infty}(I)}\le \mu$\ek{. Also}, by \ek{uniqueness} $u_h^n=\tilde{u}_h^n$. Therefore, \eqref{eq:s2.1} implies the desired conclusion. 
\end{proof}

\begin{proof}[Proof of Theorem \ref{th:s4}]
The proof follows the exact same \ek{pattern} as \ek{that} for Theorem \ref{th:s3}\ek{, but}
using Theorem \ref{th:s2} 
instead of Theorem \ref{th:s1}.  
\end{proof}

\subsection{Proof of Theorems \ref{th:s5} and \ref{th:s6} }
\label{sec:48}

We use the projection operator $P_B$ of $\dot{H}^1\to S_h$ associated with $B(\cdot,\cdot)$: 
\begin{equation}
B(P_{B}w -w,\chi)=0\qquad (\chi\in S_{h}) .
 \label{eq:pB}
\end{equation}

In \cite{et84}, the following error estimates are proved.  

\begin{lemma}
\label{prop:tj3}
Letting $w\in C^{2}(\bar{I})\cap\dot{H}^1$ and \eqref{eq:beta} \ek{be} satisfied, \ek{then} for $h\le h_8=h_8(N,\beta)$ we obtain 
\begin{equation}
\label{eq:tj3}
\|P_Bw-w\|_{L^{\infty}(I)}\le C_8 h^{2}\|w_{xx}\|_{L^{\infty}(I)},
\end{equation}
where $C_8=C_8(N,\beta)$. 
\end{lemma}

We also use a version of Poincar\'{e}'s inequality (see \cite[Lemma 18.1]{tho06}).

\begin{lemma}
\label{la:p}
We have 
\begin{equation}
\label{eq:po}
\vnorm{w}\le \vnorm{w_{x}}\qquad (w\in \dot{H}(I)).
\end{equation}
\end{lemma}

We \ek{can now} state the proof that follows. 

\begin{proof}[Proof of Theorem \ref{th:s5}]
Using $P_{B}u(t)\in S_{h}$, we decompose the error into
\[
u_{h}^{n}-u(t_{n})=\underbrace{(u_{h}^{n}- P_{B}u(t_{n}))}_{=\theta^{n}} + \underbrace{(P_{B}u(t_{n}) - u(t_{n}))}_{=\rho^{n}}. 
\]
We know from \eqref{eq:tj3} that
\begin{subequations} 
\label{eq:tj3aa}
\begin{align}
\vnorm{\rho^n}& \le \|\rho^{n}\|_{L^{\infty}(I)}  \le 
Ch^{2}\|u_{xx}\|_{L^\infty(Q_T)}, \label{eq:tj3a}\\
\vnorm{\partial_{\tau_n}\rho^{n+1}}&\le \|\partial_{\tau_n}\rho^{n+1}\|_{L^{\infty}(I)}  \le Ch^{2}\|u_{xxt}\|_{L^\infty(Q_T)}. \label{eq:tj3b}
\end{align}
\end{subequations}

\ek{Therefore}, we will \ek{specifically examine estimation of} $\vnorm{\theta^{n}_{x}}$ \ek{because} we are aware that
\[
\|\chi\|_{L^{\infty}(I)}\le\|\chi_{x}\|_{L^{1}(I)}\le C\left(\log\frac{1}{h}\right)^{\frac{1}{2}}\vnorm{\chi_{x}}\qquad (\chi\in S_{h}).
\]
Furthermore, \eqref{eq:w2} and \eqref{eq:8} give 
\begin{multline}
\label{eq:s5.1}
\dual{\partial_{\tau_n}\theta^{n+1}+\partial_{\tau_n}\rho^{n+1},\chi} 
+B(\theta^{n+1},\chi)=
\dual{f(u_{h}^{n})-f(u(t_{n})),\chi}\\
-\dual{f(u(t_{n+1}))-f(u(t_{n})),\chi}
-\dual{\partial_{\tau_n} u(t_{n+1})-u_{t}(t_{n+1}),\chi}
\end{multline}
for $\chi\in S_h$. 
Substituting this for $\chi=\theta^{n+1}$, we have 
\begin{multline}
\label{eq:s5.2}
\dual{\partial_{\tau_n}\theta^{n+1},\theta^{n+1}} 
+B(\theta^{n+1},\theta^{n+1})\\
=\dual{f(u_{h}^{n})-f(u(t_{n})),\theta^{n+1}}
-\dual{f(u(t_{n+1}))-f(u(t_{n})),\theta^{n+1}}\\
-\dual{\partial_{\tau_n} u(t_{n+1})-u_{t}(t_{n+1}),\theta^{n+1}}
-\dual{\partial_{\tau_n}\rho^{n+1},\theta^{n+1}}.
\end{multline}
This, together with \eqref{eq:bc2}, implies that
\begin{multline*}
\vnorm{\theta^{n+1}_{x}}^2\le 
M\vnorm{u_{h}^{n}-u(t_{n})} \cdot \vnorm{\theta^{n+1}} \\
+M\vnorm{u(t_{n+1})-u(t_{n})}\cdot \vnorm{\theta^{n+1}}
+\vnorm{\partial_{\tau_n} u(t_{n+1})-u_{t}(t_{n+1})}\cdot \vnorm{\theta^{n+1}}\\
+\vnorm{\partial_{\tau_n}\rho^{n+1}}\cdot \vnorm{\theta^{n+1}}
+\vnorm{\partial_{\tau_n}\theta^{n+1}}\cdot \vnorm{\theta^{n+1}}.
\end{multline*}
Therefore, using \eqref{eq:po}, we deduce \ek{that}
\begin{align}
\vnorm{\theta^{n+1}_{x}} 
& \le M\vnorm{u_{h}^{n}-u(t_{n})}  +M\vnorm{u(t_{n+1})-u(t_{n})}\nonumber \\
& { }\qquad  +\vnorm{\partial_{\tau_n} u(t_{n+1})-u_{t}(t_{n+1})}
+\vnorm{\partial_{\tau_n}\rho^{n+1}}
+\vnorm{\partial_{\tau_n}\theta^{n+1}}\nonumber\\
& \le M(\vnorm{\theta^n}+\vnorm{\rho^n})  +M\tau_n\|u_t\|_{L^\infty(Q_T)}\nonumber\\
& { }\qquad +\tau_{n} \|u_{tt}\|_{L^{\infty}(Q_T)}
+\vnorm{\partial_{\tau_n}\rho^{n+1}}
+\vnorm{\partial_{\tau_n}\theta^{n+1}}.
\label{eq:s5.4}
\end{align}
These estimates actually hold\ek{. Nevertheless}, their proof \ek{is} postponed for Appendix \ref{sec:a1}: 
\begin{subequations} 
\label{eq:a1}
\begin{align}
\vnorm{\theta^n}& \le C (h^2+\tau), \label{eq:a11}\\
\vnorm{\partial_{\tau_n}\theta^{n+1}} &\le 
C\left(h^2+\tau+\frac{\delta}{\tau}\right). \label{eq:a12}
\end{align}
\end{subequations}
Using \eqref{eq:tj3a}, 
\eqref{eq:tj3b}, 
\eqref{eq:a11}, and 
\eqref{eq:a12}, we deduce 
\[
 \vnorm{\theta^{n+1}_{x}} \le C\left(h^2+\tau+\frac{\delta}{\tau}\right),
\] 
which completes the proof of Theorem \ref{th:s5}. 
\end{proof}

Finally, we state the \ek{following} proof. 

\begin{proof}[Proof of Theorem \ref{th:s6}]
Consider problems \eqref{eq:1} and \eqref{eq:8} 
with replacement~$f(s)=s|s|^{\alpha}$ by
\[
\tilde{f}(s)=
\begin{cases}
s|s|^{\alpha}&(|s|\le\mu)\\
[(1+\alpha)\mu^{\alpha}s-\alpha\mu^{1+\alpha}] \operatorname{sgn} (s)&(|s|\ge\mu),
\end{cases}
\]
where $\mu>0$ is determined later. Then, 
$\tilde{f}$ is a $C^1$ function and the corresponding values of $\tilde{M}_1$ and $\tilde{M}_2$ in (f4) are expressed as 
$\tilde{M}_1=(1+\alpha)\mu^{\alpha}$ and $\tilde{M}_2=(1+\alpha)\alpha\mu^{\alpha-1}$. 

\ek{Let} $\tilde{u}$ and $\tilde{u}_{h}^{n}$ \ek{respectively represent} the solutions of \eqref{eq:1} and \eqref{eq:8} with $\tilde{f}$.
If $\mu\ge \kappa_1(u)$, then $u=\tilde{u}$ holds true by uniqueness.
\ek{Consequently}, we can apply Theorem \ref{th:s5} to obtain
\begin{equation}
  \|\tilde{u}_{h}^{n}-u(t_{n})\|_{L^{\infty}(I)}\le C\left(\log\frac{1}{h}\right)^{\frac{1}{2}}(h^{2}+\tau),
\label{eq:a100}
\end{equation}
where~$C=C(T,\kappa_1(u),\gamma,N,\beta)$. 
At this juncture, we apply small $h$ and $\tau$ such that $C\left(\log\frac{1}{h}\right)^{\frac{1}{2}}(h^{2}+\tau)<1$, and set $\mu=\kappa_1(u)+1$. As $\|\tilde{u}_{h}^{n}\|_{L^{\infty}(I)}\le \kappa_1(u)+1=\mu$, we obtain $\tilde{u}_{h}^{n}=u_{h}^{n}$ by the uniqueness theorem. 
Therefore, \eqref{eq:a100} implies the desired estimate.  
\end{proof}


\section{Numerical examples}
\label{sec:5}

\tn{
This section presents some numerical examples to validate our theoretical results.
For this purpose, throughout this section, we set
\[  
f(s) = s|s|^{\alpha},~\alpha>0
\]
If this were the case, then the solution of (1) might blow up in the finite time.
Therefore, one must devote particular attention to setting of the time increment $\tau_{n}$.
Particularly, following Nakagawa \cite{nak76} (see also Chen \cite{che92} and Cho--Hamada--Okamoto \cite{cho07}), we
use the time-increment control
\begin{equation}
\label{eq:6.1a} 
\tau_{n}=\tau\cdot \min\left\{1,\  
\frac{1}{\|u_{h}^{n}\|_{2}^{\alpha}}\right\}
\qquad  \left(\|u_{h}^{n}\|_{2}^2=
\textstyle\sum\limits_{j=0}^{m-1}hx_{j+1}^{N-1}u_{h}^{n}(x_j)^{2}
\right),
\end{equation}
where $\tau=\lambda h^2$ and $\lambda=1/2$. 
 
}

\begin{figure}[htbp]
      \begin{minipage}{.49\textwidth}
        \begin{center}
           \includegraphics[width=.9\textwidth]{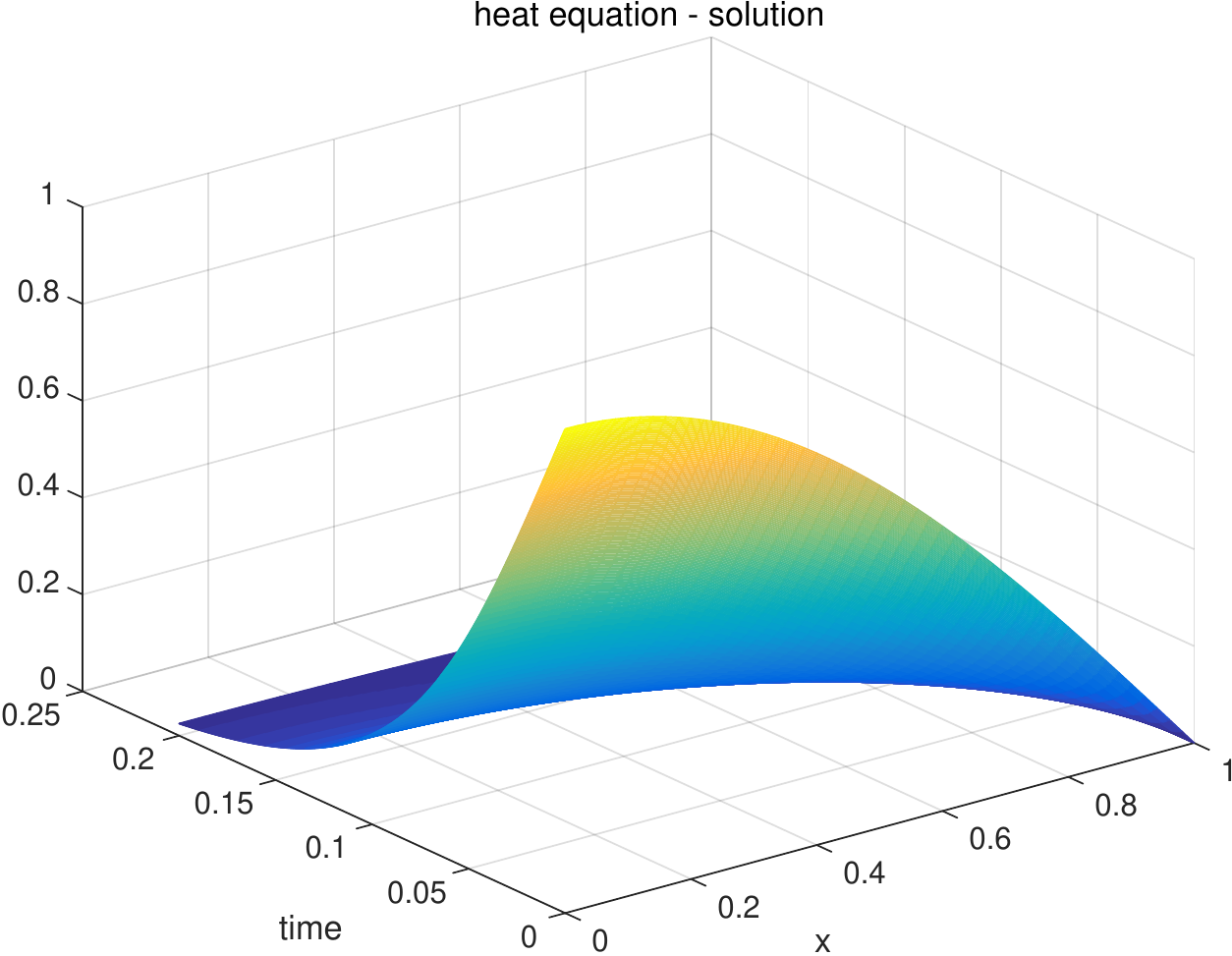} \\
         (a) (Sym)
        \end{center}
      \end{minipage}
      \begin{minipage}{.49\textwidth}
        \begin{center}
           \includegraphics[width=.9\textwidth]{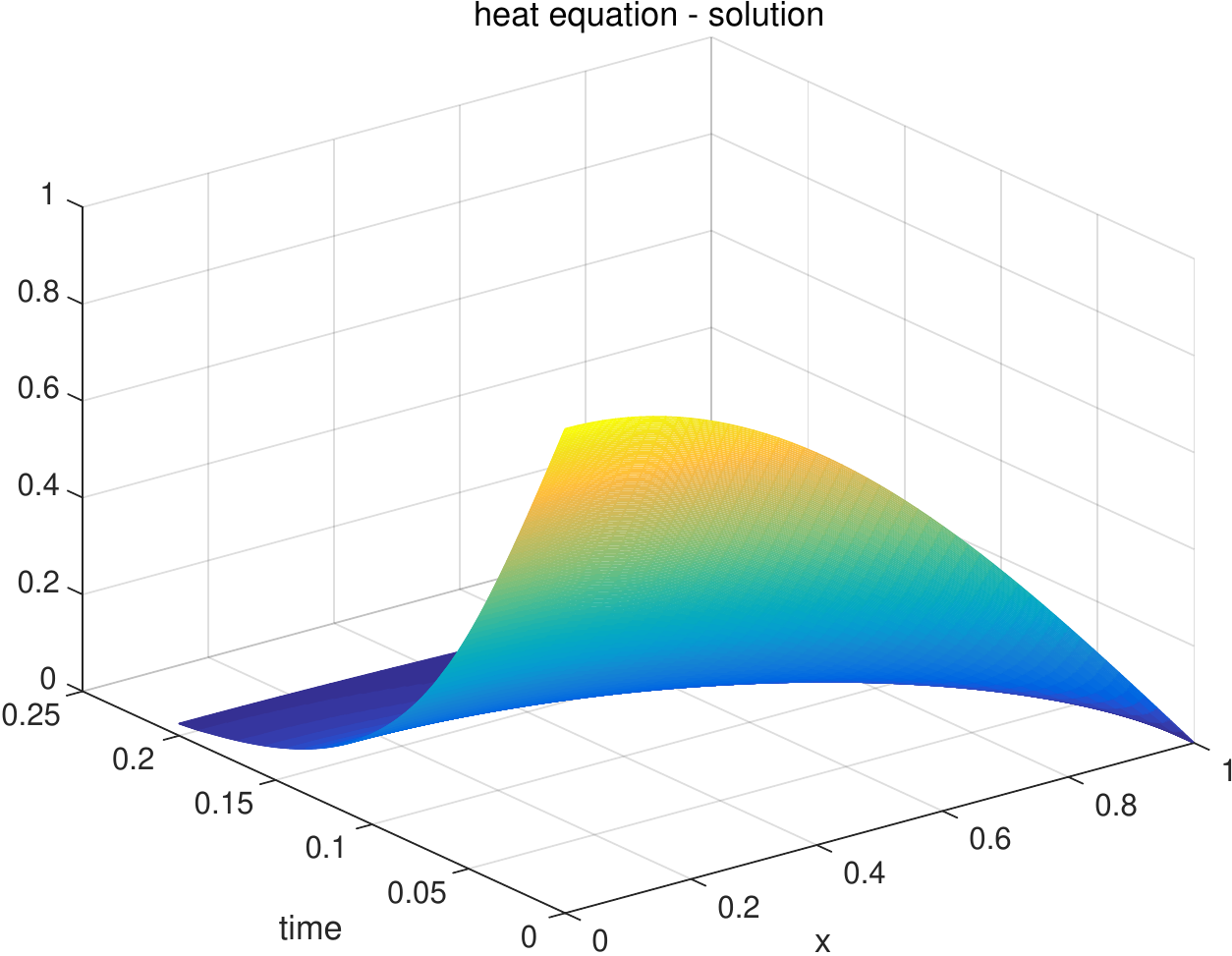}\\
          (b) (Non-Sym)
        \end{center}
      \end{minipage}
     \caption{$N=5$, $\alpha=\frac{4}{3}$ and $u(0,x)=\cos\frac{\pi}{2}x$.}
\label{fig:1}
\end{figure}

First, we compared the shapes of both solutions of (Sym) and (Non-Sym), as shown in 
Fig.~\ref{fig:1} for $N=5$, $\alpha=\frac{4}{3}$ and $u(0,x)=\cos\frac{\pi}{2}x$. We used the uniform space mesh $x_j=jh$ ($j=0,\ldots,m$) and $h=1/m$ with $m=50$.\\
We \ek{computed them} continuously until $t_{n}\tn{=} T=0.2$ or $\|u_{h}\|_{2}^{-1}<\epsilon=10^{-8}$, wherein
both solutions exist globally in time and \ek{approach} $0$ uniformly in $\overline{I}$ as $t\to\infty$. 
No \ek{marked} differences were observed in Figs.~\ref{fig:1}(a) and \ek{~\ref{fig:1}(b).} 
Subsequently, we took Fig.~\ref{fig:2} \tn{for} the case \ek{in which} the initial value was $u(0,x)=13\cos\frac{\pi}{2}x$\ek{. The} rest of the parameters are the same. At this point, the solutions of (Sym) and (Non-Sym) blew up after $x=0.06$ with the distinct observation that the solution of the former blew up earlier than that of the latter. Furthermore, the solution of (Non-Sym) had negative values \ek{whereas} that of (Sym) was always positive.

\begin{figure}[htbp]
      \begin{minipage}{.49\textwidth}
        \begin{center}
           \includegraphics[width=.9\textwidth]{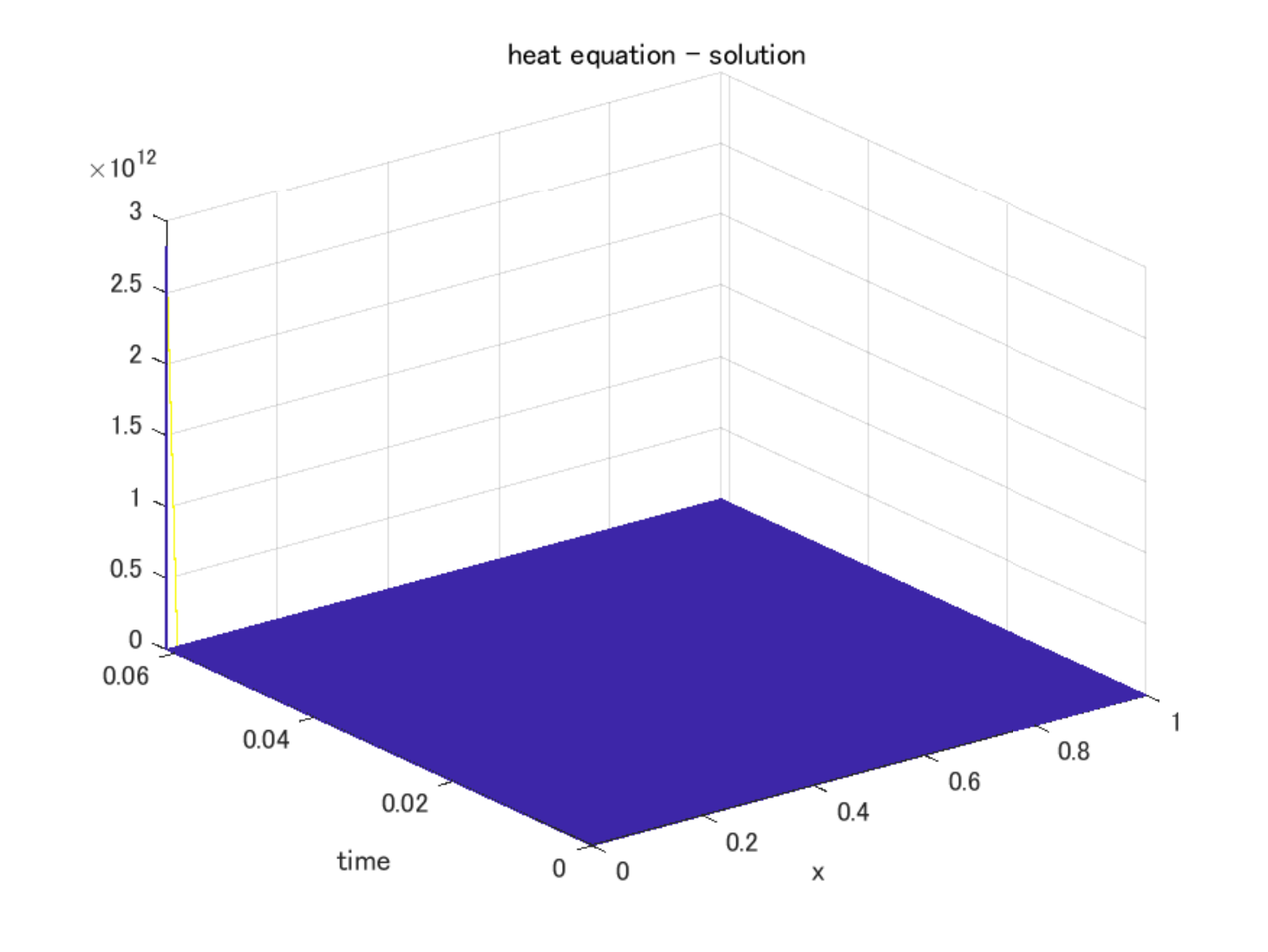} \\
         (a) (Sym)
        \end{center}
      \end{minipage}
      \begin{minipage}{.49\textwidth}
        \begin{center}
           \includegraphics[width=.9\textwidth]{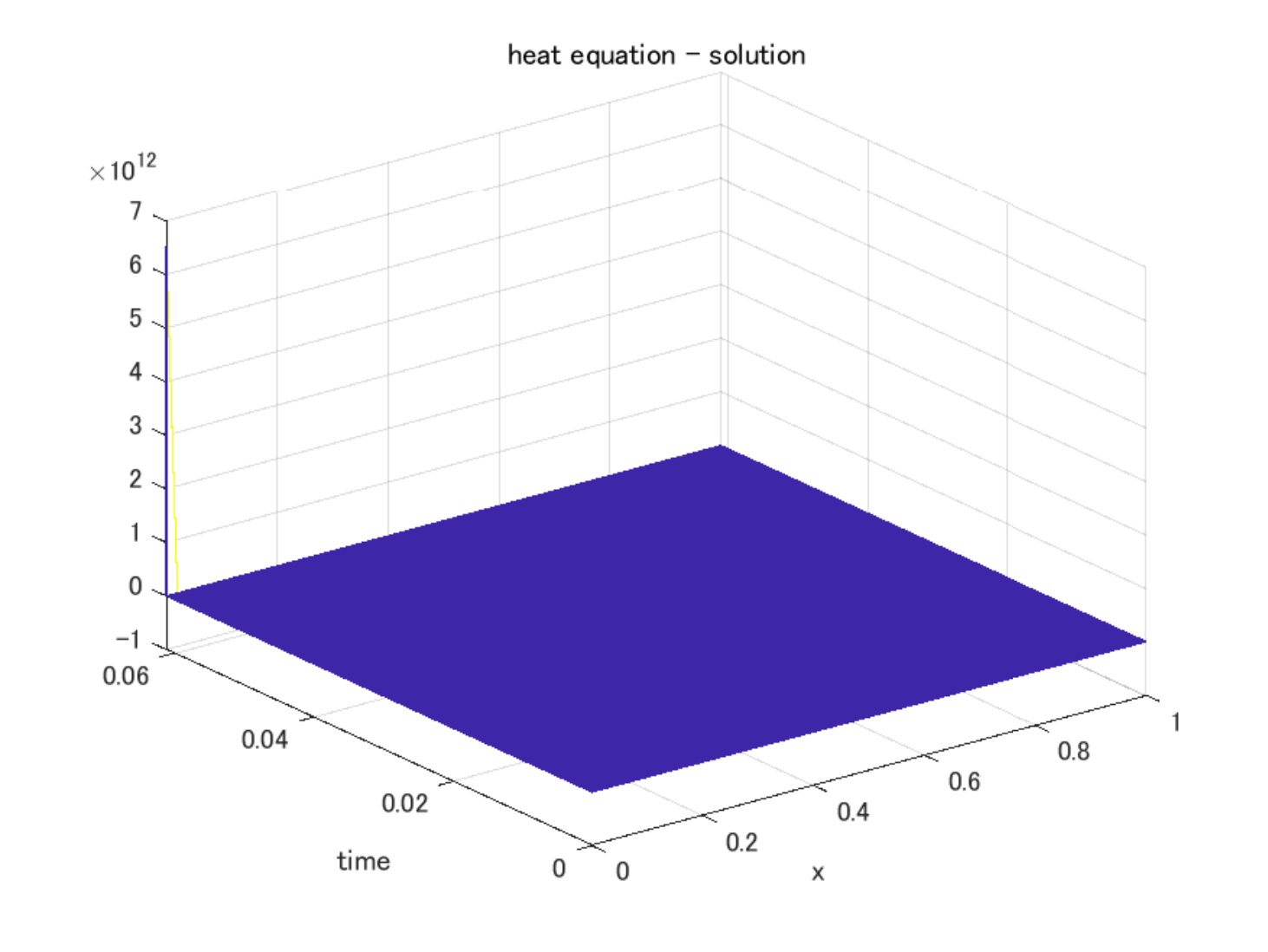}\\
         (b) (Non-Sym)
        \end{center}
      \end{minipage}
     \caption{$N=5$, $\alpha=\frac{4}{3}$ and $u(0,x)=13\cos\frac{\pi}{2}x$.}
  \label{fig:2}
\end{figure}

\ek{We} examined the error estimates of the solutions for the same uniform space mesh $x_j=jh$ ($j=0,\ldots,m$) and $h=1/m$\ek{. Also, we} regarded the numerical solution with $h'=1/480$ as the exact solution. 
The following quantities were compared: 
\begin{align*}
& \mbox{$L^{1}$err} && \|u_{h'}^{n}-u_{h}^{n}\|_{L^{1}(I)};\\
& \mbox{$L^{2}$err} && \left\|u_{h'}^{n}-u_{h}^{n}\right\|=\left\|x^{\frac{N-1}{2}}(u_{h'}^{n}-u_{h}^{n})\right\|_{L^{2}(I)};\\
& \mbox{$L^{\infty}$err} && \|u_{h'}^{n}-u_{h}^{n}\|_{L^{\infty}(I)}.
\end{align*}

Fig.~\ref{fig:3} \ek{presents} results for 
$N=3$, $\alpha=\frac{4}{3}$ and $u(0,x)=\cos\frac{\pi}{2}x$. 
We used the uniform time increment $\tau_n=\tau=\lambda h^2$ $(n=0,1,\ldots)$ with $\lambda=1/2$ and computed until $t\le T=0.005$. 
For (Sym), we observed the theoretical convergence rate $h^2+\tau$ in the $\|\cdot\|$ norm (see Theorem \ref{th:s3})\ek{,} whereas the rate in the $L^\infty$ norm \tn{deteriorated} slightly. For (Non-Sym), we observed \ek{second-order} convergence in the $L^\infty$ norm, which supports the results \ek{presented} in Theorem \ref{th:s5}.

\begin{figure}[htbp]
      \begin{minipage}{.49\textwidth}
        \begin{center}
           \includegraphics[width=.9\textwidth]{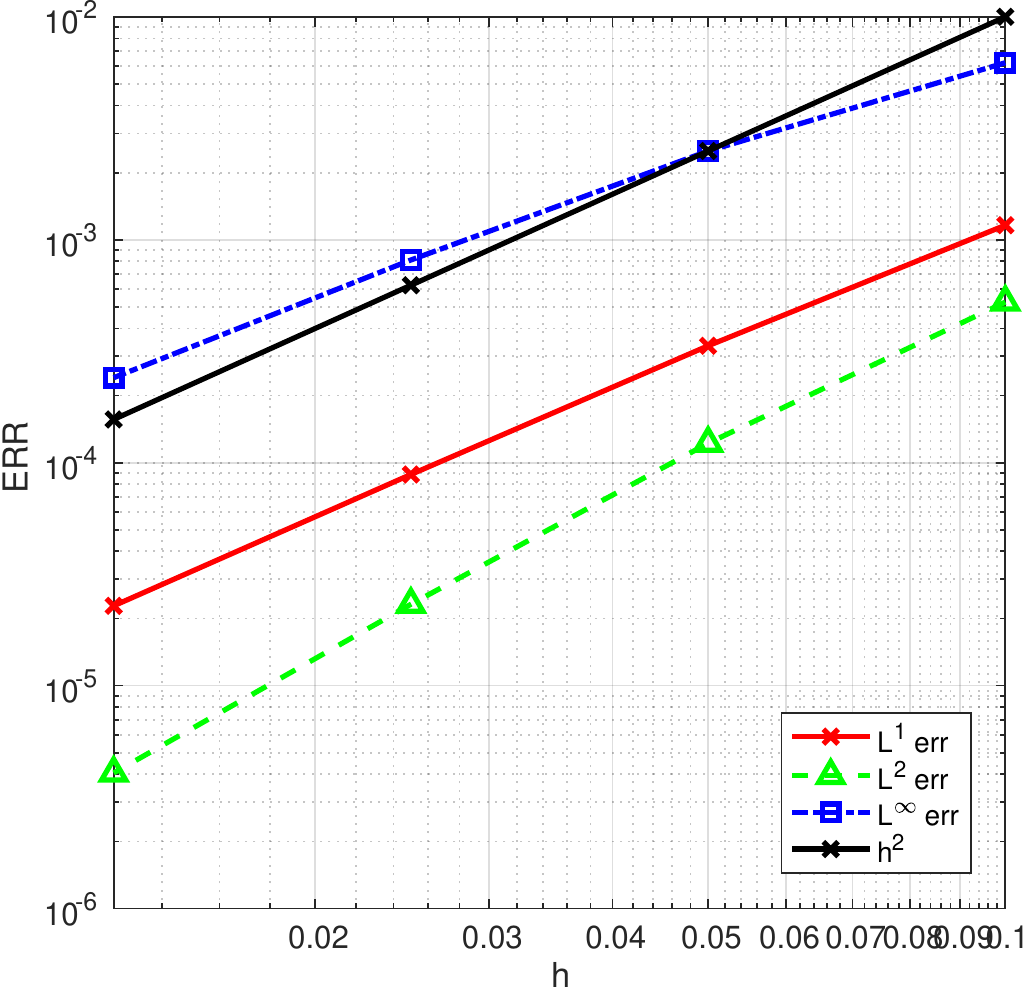} \\
         (a) (Sym)
        \end{center}
      \end{minipage}
      \begin{minipage}{.49\textwidth}
        \begin{center}
           \includegraphics[width=.9\textwidth]{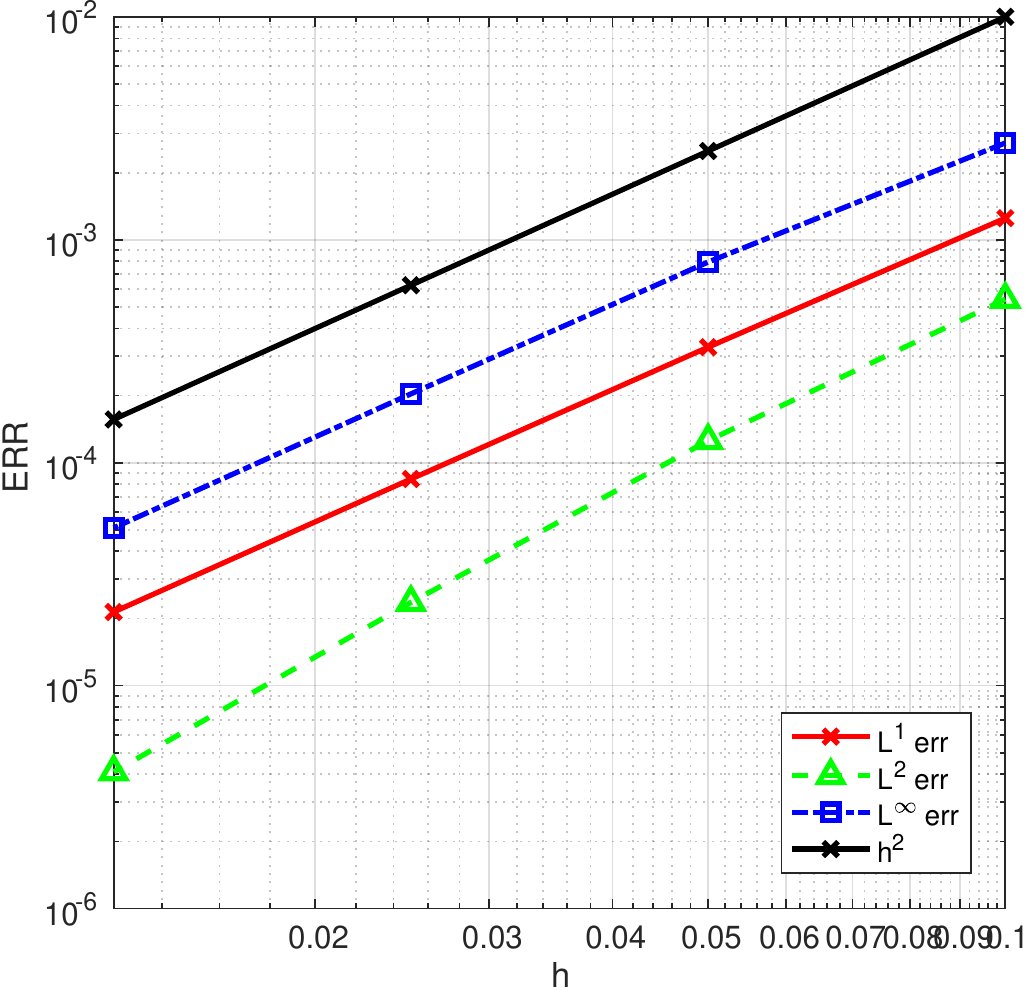}\\
          (b) (Non-Sym)
        \end{center}
      \end{minipage}
     \caption{Errors. $N=3$, $\alpha=\frac{4}{3}$ and $u(0,x)=\cos\frac{\pi}{2}x$.}
\label{fig:3}
\end{figure}

Moreover, we considered the case for $N=4$, which is not supported in Theorem \ref{th:s3} for (Sym)\ek{. Also, we} chose $\alpha=4$ and $u(0,x)=3\cos\frac{\pi}{2}x$ for this case. Fig.~\ref{fig:4}(d) displays the shape of the solution, which blew up at approximately $T=0.0035$. 
Furthermore, we computed errors until $T=0.0011, 0.0022$, and $0.0033$ using the uniform meshes $x_j$ and $\tau_n$ with $\lambda=0.11$. 
 From Fig.~\ref{fig:4}, we observed the second-order convergence in the $\|\cdot\|$ norm, suggesting the possibility of removing assumption \tn{$N\le 3$}.

 \begin{figure}[htbp]
      \begin{minipage}{0.49\textwidth}
        \begin{center}
           \includegraphics[width=.9\textwidth]{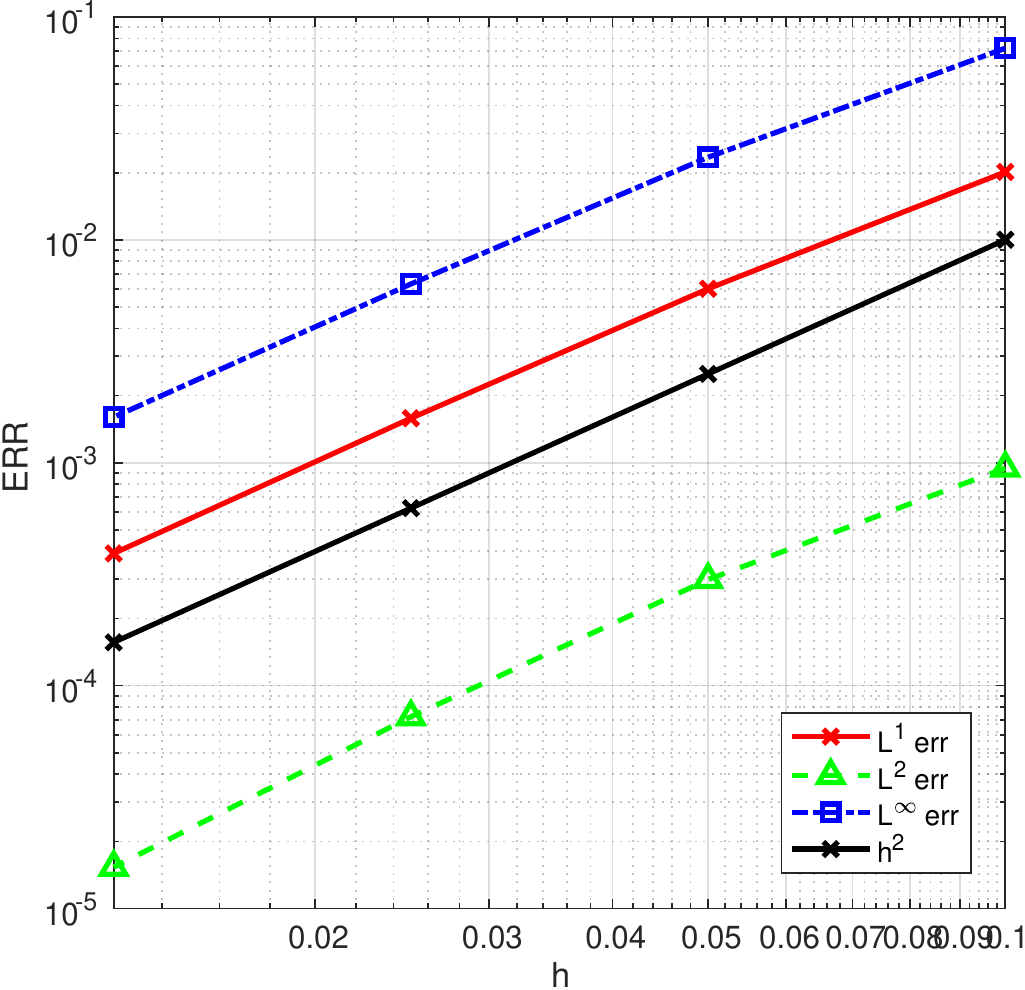} \\
(a) $T=0.0011$ 
        \end{center}
      \end{minipage}
      \begin{minipage}{0.49\textwidth}
        \begin{center}
           \includegraphics[width=.9\textwidth]{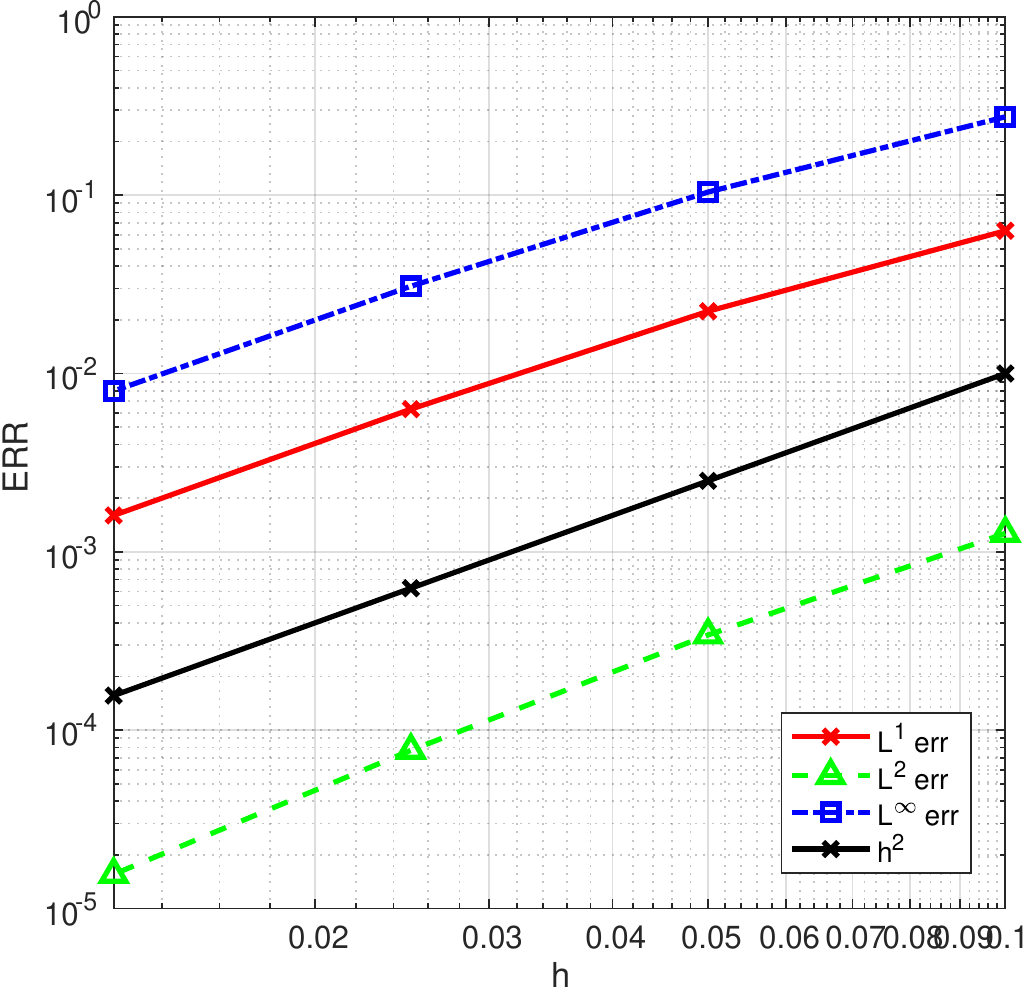}\\
          (b) $T=0.0022$
        \end{center}
      \end{minipage}
      \begin{minipage}{0.49\textwidth}
        \begin{center}
           \includegraphics[width=.9\textwidth]{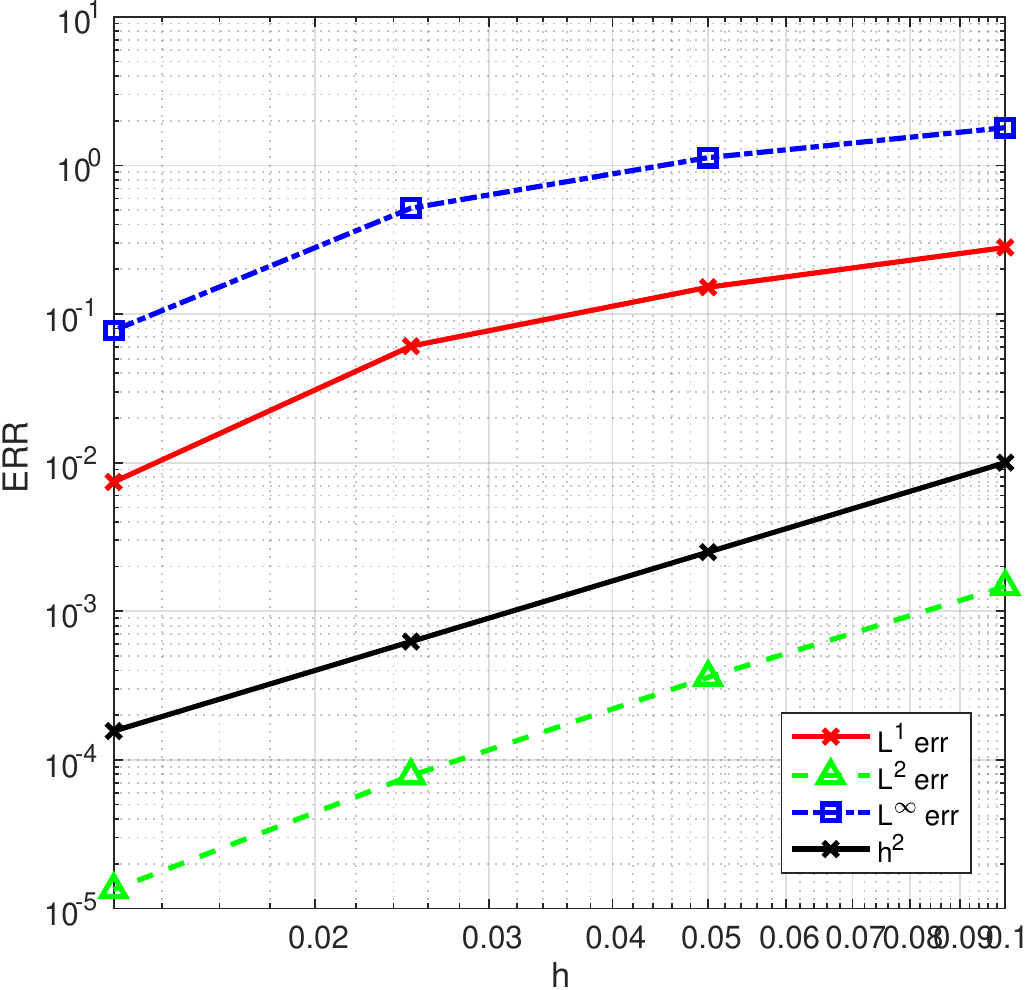}\\
(c) $T=0.0033$
        \end{center}
      \end{minipage}
       \begin{minipage}{0.49\textwidth}
        \begin{center}
           \includegraphics[width=.9\textwidth]{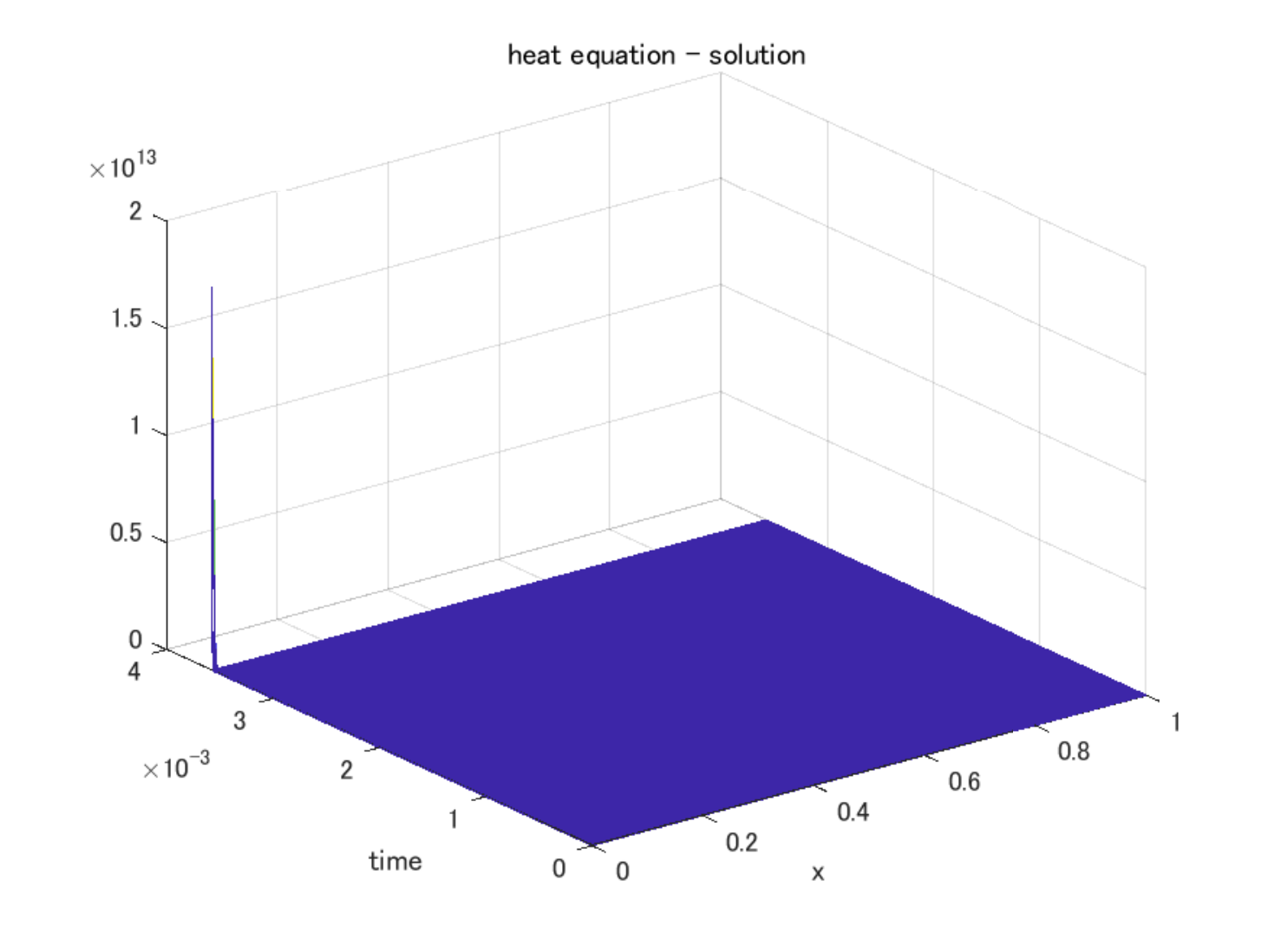}\\
      (d) solution shape
        \end{center}
      \end{minipage}
     \caption{Errors. $N=4$, $\alpha=4$ and $u(0,x)=3\cos\frac{\pi}{2}x$.}
\label{fig:4}
\end{figure}
 
Finally, we \tn{observed the non-increasing property of the energy functional. The energy functional associated with (1) is given as
\[
J(t)=\frac{1}{2}\|u_{x}\|^2-\frac{1}{\alpha+2}\int_I x^{N-1}|u|^{\alpha+2}~dx.
\]
We can use the standard method to prove that $J(t)$ is non-increasing in $t$.

This non-increasing property plays an important role in the blow-up analysis of the solution of (1)\ek{, as presented by} Nakagawa \cite{nak76}.
Therefore, it is of interest whether a discrete version of this non-increasing property holds true. 
Actually, introducing the discrete energy functional associated with (Sym) as 
\[
 J_{h}(n)=\frac{1}{2}\|(u^{n}_{h})_x\|^{2}-\frac{1}{\alpha+2}\int_Ix^{N-1}|u_{h}^{n}|^{\alpha+2}~dx,
\]
we prove the following. \ek{Appendix B presents the proof.}
\begin{prop}
\label{prop:5.5}
$J_{h}(n)$ is a non-increase sequence of $n$. 
\end{prop}

}

 Now let $N=3$, $\alpha=\frac{4}{3}$, and $u(0,x)=\cos\frac{\pi}{2}x,~13\cos\frac{\pi}{2}x$. We determined the time increment $\tau_{n}$ through \eqref{eq:6.1a} \ek{for} the uniform space mesh $x_j=jh$ with $h=1/m$ and \tn{$m=50$}. 
Fig.~\ref{fig:6} \ek{presents} the results, which support that of Proposition \ref{prop:5.5}. 

\begin{figure}[htbp]
      \begin{minipage}{0.49\textwidth}
        \begin{center}
           \includegraphics[width=0.9\textwidth]{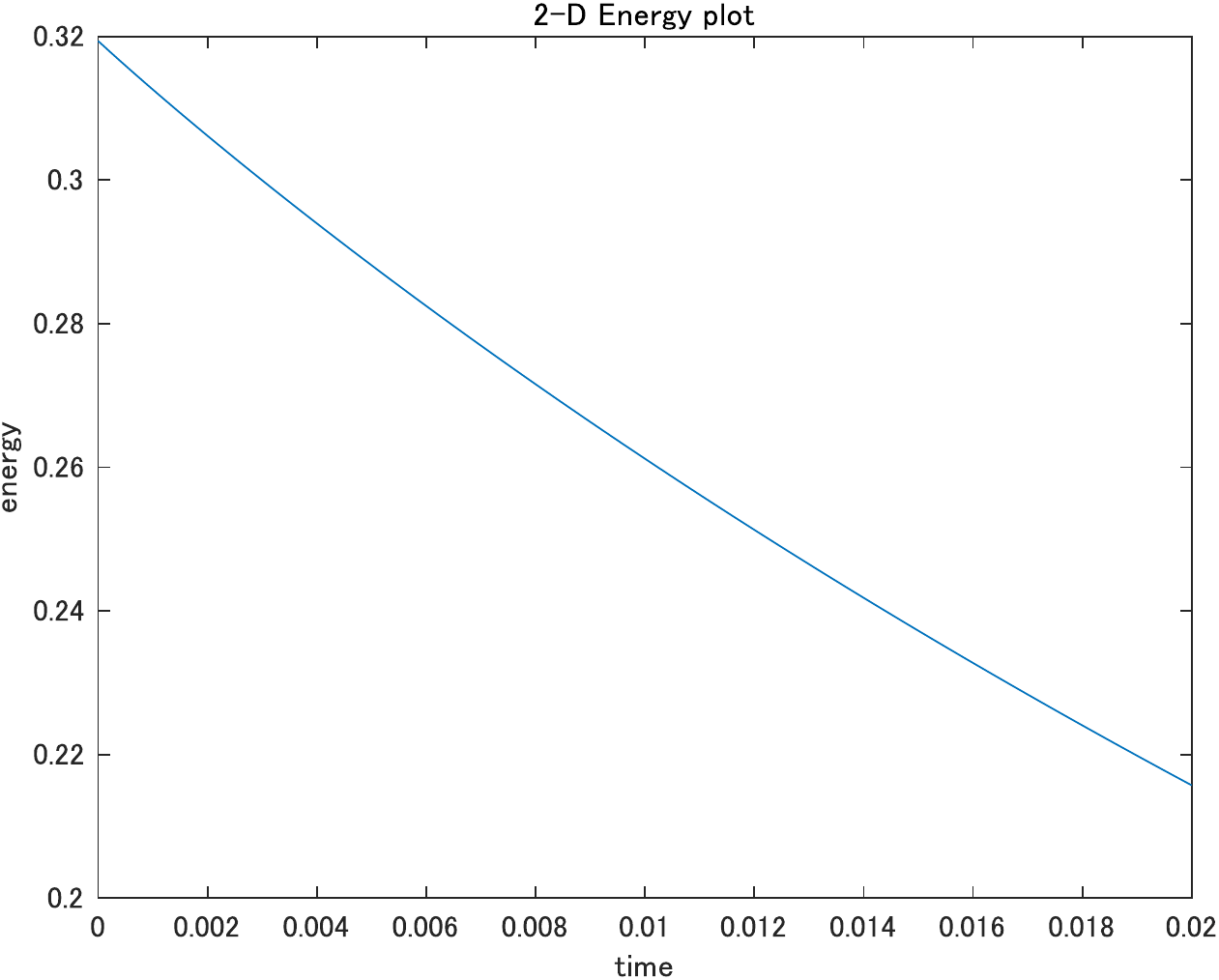} \\
         (Sym) \& $u(0,x)=\cos\frac{\pi}{2}x$
        \end{center}
      \end{minipage}
      \begin{minipage}{0.49\textwidth}
        \begin{center}
           \includegraphics[width=0.9\textwidth]{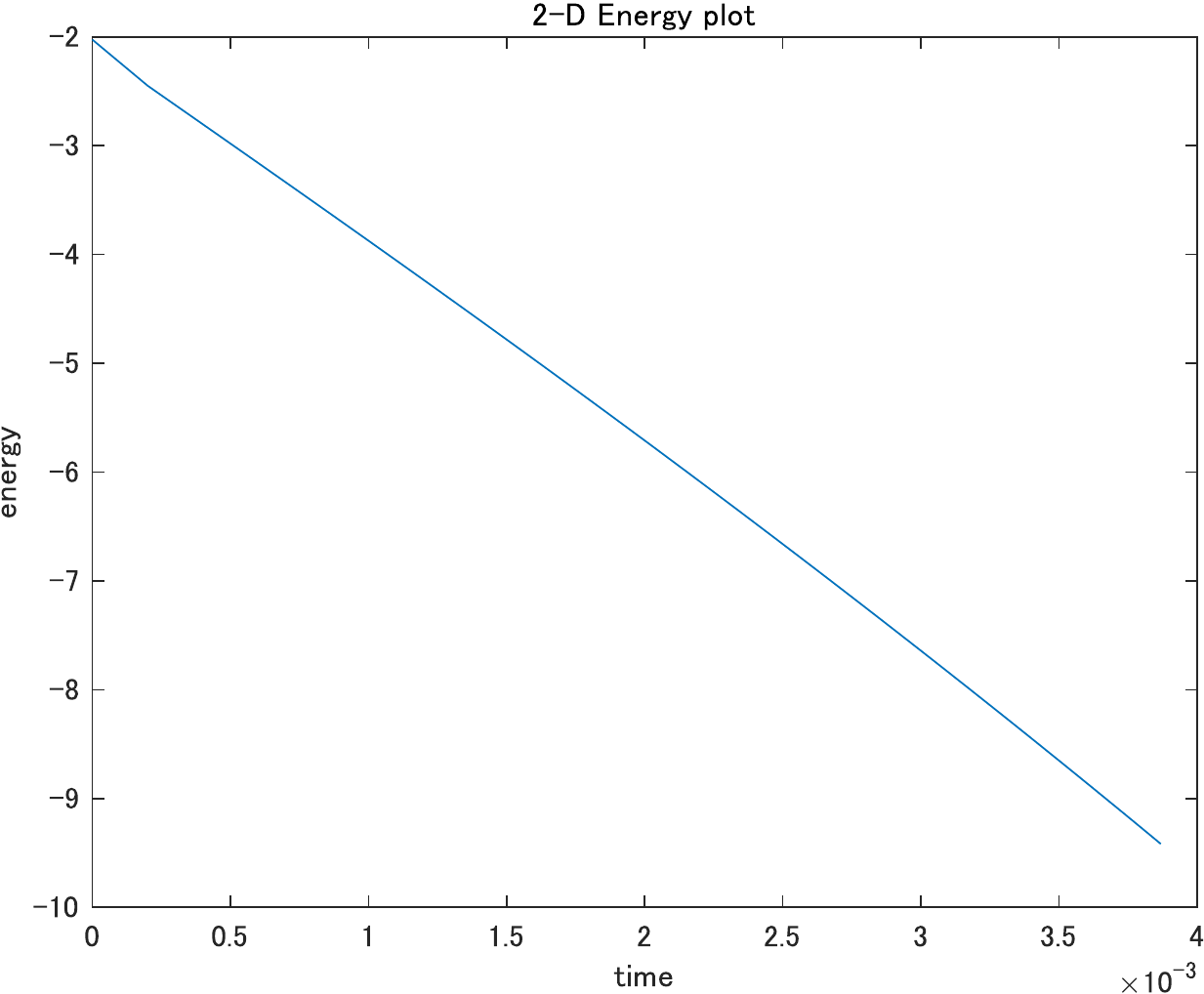}\\
         (Sym) \& $u(0,x)=13\cos\frac{\pi}{2}x$
        \end{center}
      \end{minipage}
     \caption{Energy functional.}
\label{fig:6}
\end{figure}

\appendix
\section{Proofs of \eqref{eq:a11} and \eqref{eq:a12}} 
\label{sec:a1}

Proofs of \eqref{eq:a11} and \eqref{eq:a12} are stated in this appendix \ek{using the same notation as that} used in Section \ref{sec:4}.

\begin{proof}[Proof of \eqref{eq:a11}]
\ek{By application of} \eqref{eq:s5.2}, \eqref{eq:tj3a}, and \eqref{eq:tj3b}, we derived the expression 
\begin{multline*}
\frac{1}{\tau_n}\left(\vnorm{\theta^{n+1}}^2-\vnorm{\theta^{n+1}}\cdot 
\vnorm{\theta^{n}}\right)
\le 
M(\vnorm{\theta^n}+Ch^2\|u_{xx}\|_{L^\infty(Q_T)})\cdot \vnorm{\theta^{n+1}}\\
+M\tau_n \|u_t\|_{L^\infty(Q_T)}\cdot\vnorm{\theta^{n+1}}
+\tau_{n} \|u_{tt}\|_{L^{\infty}(Q_T)}\cdot\vnorm{\theta^{n+1}}\\
+Ch^2\|u_{xxt}\|_{L^\infty(Q_T)}\cdot\vnorm{\theta^{n+1}}.
\end{multline*}
Consequently, we have
\[
\vnorm{\theta^{n+1}}
\le 
(1+\tau_n M) \vnorm{\theta^n}+C\tau_n (h^2+\tau_n).
\]
Therefore, \ek{similarly to} the derivation of \eqref{eq:th1.14}, we obtain from \eqref{eq:iv3} the expression \ek{of}
\[
 \vnorm{\theta^n}\le C\left(h^2+\tau\right)
\]
to complete the proof. 
\end{proof}

\begin{proof}[Proof of \eqref{eq:a12}]
First, we prove the case \ek{of} $n=0$. 
Substituting \eqref{eq:s5.1} for $n=0$ and $\chi=\theta^{1}$, we obtain 

\begin{multline*}
\dual{\frac{\theta^{1}-\theta^0}{\tau_{0}},\theta^{1}}+B(\theta^{1},\theta^{1})
\le
\dual{f(u_{h}^{0})-f(u^0),\theta^1} \\
 -\dual{f(u(t_{1}))-f(u^0),\theta^1}
-\dual{\partial_{\tau_0} u(t_{1})-u_{t}(t_{1}),\theta^1}
-\dual{\frac{\rho^{1}-\rho^0}{\tau_{0}},\theta^{1}}.
\end{multline*}
\ek{Because} $\theta^0=0$, we apply \eqref{eq:tj3b} \ek{to} get
\begin{align*}
\frac{1}{\tau_0}\vnorm{\theta^{1}}^2
& \le
M \vnorm{\rho^0}\cdot \vnorm{\theta^1}
+M\tau_0\|u_t\|_{L^\infty(Q_T)}\vnorm{\theta^1}\\
&{ }\quad +\tau_0\|u_{tt}\|_{L^\infty(Q_T)}\vnorm{\theta^1}
+\vnorm{\partial_{\tau_0}\rho^{1}}\cdot \vnorm{\theta^1}\\
&\le  C(\tau_0+h^2)\vnorm{\theta^1}.
\end{align*}
Repeatedly using $\theta^0=0$, we obtain 
\begin{equation}
 \vnorm{\partial_{\tau_0}\theta^1}\le C(\tau_0+h^2).
\label{eq:a12.00}
\end{equation}

\ek{Next} we assume $n\ge 0$ and $t_{n+2}\le T$.  
\ek{Consequently}, from \eqref{eq:s5.1}, we derive 
\begin{align}
& \dual{\partial_{\tau_{n+1}}\theta^{n+2}-\partial_{\tau_n}\theta^{n+1},\chi} 
+B(\theta^{n+2}-\theta^{n+1},\chi) \nonumber \\
&\mbox{ }\quad
=\langle \underbrace{f(u_{h}^{n+1})-f(u(t_{n+1}))-f(u_{h}^{n})+f(u(t_{n}))}_{=J_1},\chi\rangle \nonumber \\
&\mbox{ }\qquad -\langle \underbrace{f(u(t_{n+2}))-f(u(t_{n+1}))-f(u(t_{n+1}))+f(u(t_{n}))}_{=J_2},\chi \rangle  \nonumber \\
&\mbox{ }\qquad -\langle \underbrace{\partial_{\tau_{n+1}}u(t_{n+2})-u_{t}(t_{n+2}) -\partial_{\tau_{n}}u(t_{n+1})
 + u_{t}(t_{n+1})}_{=J_3} ,\chi\rangle  \nonumber \\
&\mbox{ }\qquad -\langle \underbrace{\partial_{\tau_{n+1}}\rho^{n+2}-\partial_{\tau_{n}}\rho^{n+1}}_{=J_4},\chi\rangle
\label{eq:a12.1}
\end{align}
for any $\chi\in S_h$. 
Substituting this \ek{expression} for $\chi=\partial_{\tau_{n+1}}\theta^{n+2}$, we \ek{obtain} 
\begin{equation*}
  \vnorm{\partial_{\tau_{n+1}}\theta^{n+2}}^2-
\vnorm{\partial_{\tau_n}\theta^{n+1}}\cdot \vnorm{\partial_{\tau_{n+1}}\theta^{n+2}} 
\le \\
\vnorm{\partial_{\tau_{n+1}}\theta^{n+2}}\sum_{j=1}^4\vnorm{J_j}.
\end{equation*}
\ek{Here}, we \ek{accept} the following estimates: 
\begin{subequations}
\begin{align}
\vnorm{J_1} &\le C\tau_n(1+\tau_n) \vnorm{\partial_{\tau_n}\theta^{n+1}} +C\tau_n(h^2+\tau_n+\tau_nh^2), \label{eq:a12.J1}\\
\vnorm{J_2},\vnorm{J_3} &\le C\tau_{n+1}(\tau_{n+1}+\tau_{n})+
C|\tau_{n+1}-\tau_{n}|, \label{eq:a12.J2}\\
\vnorm{J_4} &\le C (\tau_{n+1}+\tau_n)h^2. \label{eq:a12.J4}
\end{align} 
\end{subequations}
In view of the quasi-uniformity of time partition \eqref{eq:qt}, we have
\[
 \tau_{n+1}=\tau_n\frac{\tau_{n+1}}{\tau_n}\le \gamma \tau_n.
\] 
Summing up, we deduce
\begin{equation}
b_{n+1}-b_n\le C \tau_n b_n +C\tau_n\left(h^2+\tau+\frac{\delta}{\tau_{\min}}\right) ,
\label{eq:a12.3}
\end{equation}
where $b_n=\vnorm{\partial_{\tau_n}\theta^{n+1}}$. Therefore,  
\[
 b_{n}\le e^{CT}b_0+C(e^{CT}-1)\left(h^2+\tau+\frac{\delta}{\tau_{\min}}\right),
\]
which, together with \eqref{eq:a12.00}, implies the desired inequality \eqref{eq:a12}. 

\medskip

We now prove \eqref{eq:a12.J1}--\eqref{eq:a12.J4}. 

\smallskip

\noindent \emph{Estimation for $J_1$. } We apply Taylor's theorem to obtain 
\begin{align*}
J_1 
&=f'(s_{1})(u_{h}^{n+1}-u_{h}^{n})-f'(s_{2})(u(t_{n+1})-u(t_{n}))\\
&= f'(s_{1})[(\theta^{n+1}+\rho^{n+1})-(\theta^{n}+\rho^{n})]+\frac{f'(s_{1})-f'(s_{2})}{s_{1}-s_{2}}(s_{1}-s_{2})(u(t_{n+1})-u(t_{n})),
\end{align*}
where 
$s_{1}=u_{h}^{n+1}-\mu_{1}(u_{h}^{n+1}-u_{h}^{n})$ and $s_{2}=u(t_{n+1})-\mu_{2}[u(t_{n+1})-u(t_{n})]$ for some $\mu_{1},\mu_2\in [0,1]$. 
In view of \eqref{eq:tj3a}, \eqref{eq:tj3b}, and \eqref{eq:a11}, we find the \ek{following} estimates

\begin{align*}
\vnorm{J_1}
&\le \tau_nM\vnorm{\partial_{\tau_n}\theta^{n+1}} +\tau_nM\vnorm{\partial_{\tau_n}\rho^{n+1}}+  \wnorm{\frac{f'(s_{1})-f'(s_{2})}{s_{1}-s_{2}}(s_{1}-s_{2})} \cdot \tau_n\|u_t\|_{L^\infty(Q_T)},\\
&\le \tau_nM\vnorm{\partial_{\tau_n}\theta^{n+1}} +C\tau_n Mh^2\|u_{txx}\|_{L^\infty(Q_T)}+ \wnorm{\frac{f'(s_{1})-f'(s_{2})}{s_{1}-s_{2}}(s_{1}-s_{2})} \cdot \tau_n\|u_t\|_{L^\infty(Q_T)},
\end{align*}
and
\begin{align*}
&\wnorm{\displaystyle{\frac{f'(s_{1})-f'(s_{2})}{s_{1}-s_{2}}(s_{1}-s_{2})}} &\\
 \le& M_{2}\vnorm{\theta^{n+1}+\rho^{n+1}-\mu_{1}(\theta^{n+1}+\rho^{n+1}-\theta^{n}-\rho^{n})+(\mu_{2}-\mu_{1})(u(t_{n+1})-u(t_{n}))}&\\
 \le& M_{2}\{\vnorm{\theta^{n+1}}+\vnorm{\rho^{n+1}}+ \tau_n\vnorm{\partial_{\tau_n}\theta^{n+1}} +\tau_n\vnorm{\partial_{\tau_n}\rho^{n+1}}+\tau_n\|u_t\|_{L^\infty(Q_T)}\}&\\
 \le& M_{2}\{C(h^2+\tau)+Ch^2\|u_{xx}\|_{L^\infty(Q_T)}+ \tau_n\vnorm{\partial_{\tau_n}\theta^{n+1}} +C\tau_n h^2\|u_{txx}\|_{L^\infty(Q_T)}+\tau_n\|u_t\|_{L^\infty(Q_T)}\}.
\end{align*}

\smallskip

\noindent \emph{Estimation for $J_2$. } We begin with
\begin{align*}
J_2
&=f'(s_{3})(u(t_{n+2})-u(t_{n+1}))-f'(s_{4})(u(t_{n+1})-u(t_{n}))\\
&=\frac{f'(s_{3})-f'(s_{4})}{s_{3}-s_{4}}(s_{3}-s_{4})\tau_{n+1}u_{t}(\eta_{1})+f'(s_{4})(\tau_{n+1}u_{t}(\eta_{1})-\tau_{n}u_{t}(\eta_{2}))\\
&=\frac{f'(s_{3})-f'(s_{4})}{s_{3}-s_{4}}(s_{3}-s_{4})\tau_{n+1}u_{t}(\eta_{1}) \\
&{ }\quad +f'(s_{4})\tau_{n+1}(u_{t}(\eta_{1})-u_{t}(\eta_{2}))
+f'(s_{4})(\tau_{n+1}-\tau_n)u_{t}(\eta_{2}),
\end{align*}
where $s_{3}=u(t_{n+1})+\mu_{3}(u(t_{n+2})-u(t_{n+1}))$ and 
$s_{4}=u(t_{n+1})+\mu_{4}(u(t_{n})-u(t_{n+1}))$ for some $\mu_3,\mu_4\in [0,1]$, $\eta_1\in [t_{n+1},t_{n+2}]$, and $\eta_2\in [t_{n},t_{n+1}]$. 
Next, we obtain the \ek{following} estimate\ek{:}

\begin{align*}
\vnorm{J_2}
&\le \tau_{n+1} \wnorm{\frac{f'(s_{3})-f'(s_{4})}{s_{3}-s_{4}}(s_{3}-s_{4})}\cdot \|u_{t}\|_{L^\infty(Q_T)} \\
&{ }\quad +M_1\tau_{n+1}(\tau_{n+1}+\tau_{n})\|u_{tt}\|_{L^\infty(Q_T)}
+M_1|\tau_{n+1}-\tau_n|\cdot \|u_t\|_{L^\infty(Q_T)};
\end{align*}
\begin{align*}
\wnorm{\frac{f'(s_{3})-f'(s_{4})}{s_{3}-s_{4}}(s_{3}-s_{4})}&\le CM_{2}(\tau_{n+1}+\tau_n)\|u_t\|_{L^\infty(Q_T)}.
\end{align*}

\smallskip

\noindent \emph{Estimation for $J_3$. } 
We express $J_3$ as 
\begin{align*}
J_3
&= \frac{\tau_{n+1}u_t(t_{n+2})-\frac{1}{2}\tau_{n+1}^2u_{tt}(s_5)}{\tau_{n+1}}
-u_{t}(t_{n+2}) \\
&{ }\qquad -\left(\frac{\tau_{n}u_t(t_{n+1})-\frac{1}{2}\tau_{n}^2u_{tt}(s_6)}{\tau_{n}}
-u_{t}(t_{n+1})\right)\\
&= -\frac{1}{2}\tau_{n+1}u_{tt}(s_5) +\frac{1}{2}\tau_{n}u_{tt}(s_6)\\
&= \frac{1}{2}\tau_{n+1}(u_{tt}(s_6)-u_{tt}(s_5))- \frac{1}{2}(\tau_{n+1}-\tau_{n})u_{tt}(s_6)\\
&=\frac12 \tau_{n+1}u_{ttt}(s_{7})(s_{5}-s_{6})-\frac12 (\tau_{n+1}-\tau_{n})u_{tt}(s_{6})
\end{align*}
for some 
$s_5\in [t_{n+1},t_{n+2}]$, 
$s_6\in [t_{n},t_{n+1}]$ and $s_7\in [s_6,s_5]\subset [t_{n},t_{n+2}]$. Therefore, 
\[
\vnorm{J_3}\le \frac{1}{2}\tau_{n+1}(\tau_{n+1}+\tau_{n})\|u_{ttt}\|_{L^\infty(Q_T)}+
\frac{1}{2}|\tau_{n+1}-\tau_{n}|\cdot   \|u_{tt}\|_{L^\infty(Q_T)}.
\]

\smallskip

\noindent \emph{Estimation for $J_4$. } For some 
$s_8\in [t_{n+1},t_{n+2}]$, 
$s_9\in [t_{n},t_{n+1}]$\ek{,} and  
$s_{10}\in [s_9,s_8]$, we \ek{obtain} the expression 
\[
\frac{\rho^{n+2}-\rho^{n+1}}{\tau_{n+1}}-\frac{\rho^{n+1}-\rho^{n}}{\tau_{n}}
=\rho_{t}(s_{8})-\rho_{t}(s_{9})
=(s_8-s_9)\rho_{tt}(s_{10})
\]
Therefore, using \eqref{eq:tj3}, 
\[
 \vnorm{J_4}\le C (\tau_{n+1}+\tau_n)h^2\|u_{ttxx}\|_{L^\infty(Q_T)}.
\]

\end{proof}                                             

\section{Proof of Proposition 5.1}
\tn{
\begin{proof}
Substituting $\chi=\partial_{\tau_n}u_{h}^{n+1}$ for \eqref{eq:3}, we have 
\begin{equation*}
\|\partial_{\tau_n}u_{h}^{n+1}\|_h^{2}=-\left((u^{n+1}_{h})_x,\frac{(u^{n+1}_{h})_x-(u^{n}_{h})_x}{\tau_{n}}\right)+\left(u_{h}^{n}|u_{h}^{n}|^{\alpha},\frac{u_{h}^{n+1}-u_{h}^{n}}{\tau_{n}}\right).
\end{equation*}
Therefore, for the conditions 
\begin{subequations}
\begin{align}
&\left((u^{n+1}_{h})_x,\frac{(u^{n+1}_{h})_x-(u^{n}_{h})_x}{\tau_{n}}\right)\ge\frac{1}{2}\left((u^{n}_{h})_x+(u^{n+1}_{h})_x,\frac{(u^{n+1}_{h})_x-(u^{n}_{h})_x}{\tau_{n}}\right),\label{eq:5.5b} \\
&\left(u_{h}^{n}|u_{h}^{n}|^{\alpha},\frac{u_{h}^{n+1}-u_{h}^{n}}{\tau_{n}}\right)\le
\frac{1}{\tau_{n}(\alpha+2)}\left[\int_Ix^{N-1} (|u_{h}^{n+1}|^{\alpha+2}-|u_{h}^{n}|^{\alpha+2})~dx\right] \label{eq:5.5c},
\end{align}
\end{subequations}
we obtain 
\begin{equation}
 \left\|\partial_{\tau_n}u_{h}^{n+1}\right\|_h^{2}
\le
-\frac{1}{\tau_{n}}(J_{h}(n+1)-J_{h}(n)),
\label{eq:5.5e} 
\end{equation}
which implies that $J_{h}(n+1)\le J_{h}(n)$. 

We can validate \eqref{eq:5.5b} and \eqref{eq:5.5c}\ek{. Also,} \eqref{eq:5.5b} is derived readily. To prove \eqref{eq:5.5c}, we set $g(s)=\frac{1}{\alpha+2}|s|^{\alpha+2}$, and apply the mean value theorem to deduce 
\[
 g(u_h^{n+1})- g(u_h^{n})=w|w|^{\alpha}(u_h^{n+1}-u_h^{n}),
\]
where $w=w(x)=u_{h}^{n}+\sigma(u_{h}^{n+1}-u_{h}^{n})$ and $\sigma=\sigma(x)\in (0,1)$. Consequently, 
\begin{multline*}
J\equiv \frac{1}{\tau_{n}(\alpha+2)}\left[\int_Ix^{N-1} (|u_{h}^{n+1}|^{\alpha+2}-|u_{h}^{n}|^{\alpha+2})~dx\right] 
- \int_Ix^{N-1}u_{h}^{n}|u_{h}^{n}|^{\alpha}\frac{u_{h}^{n+1}-u_{h}^{n}}{\tau_{n}}~dx\\
= \frac{1}{\tau_n}\int_Ix^{N-1}\left[w|w|^\alpha-u_{h}^{n}|u_{h}^{n}|^{\alpha}\right](u_h^{n+1}-u_h^{n})~dx. 
\end{multline*}
\ek{Then we} repeat the mean value theorem to resolve 
\[
 w|w|^\alpha-u_{h}^{n}|u_{h}^{n}|^{\alpha}=(\alpha+1)|\tilde{w}|^\alpha(w-u^n_h)=(\alpha+1)|\tilde{w}|^\alpha\tilde{\sigma}(u^{n+1}_h-u_h^n),
\]
where $\tilde{w}=u^n_h+\tilde{\sigma}(w-u_h^n)$ and $\tilde{\sigma}=\tilde{\sigma}(x)\in (0,1)$. Therefore, 
\[
 J=\frac{1}{\tau_n}\int_I
x^{N-1}(\alpha+1)|\tilde{w}|^\alpha\tilde{\sigma}
(u_h^{n+1}-u_h^{n})^2~dx\ge 0,
\]
which gives \eqref{eq:5.5c}. 
\end{proof}
}

\paragraph{Acknowledgments.}
This work was supported by JST CREST Grant \ek{No.} JPMJCR15D1, Japan, and JSPS KAKENHI Grant \ek{No.} 15H03635, Japan. In addition, the first author was supported by the Program for Leading Graduate Schools, MEXT, Japan.


\begin{thebibliography}{1}

\bibitem{akr03}
Akrivis,~G.D., Dougalis,~V.A., Karakashian,~O.A., McKinney,~W.R.:
\newblock Numerical approximation of blow-up of radially symmetric solutions of 
the nonlinear Schr{\"{o}}dinger equation.
\newblock SIAMJ. Sci. Comput. 25, 186--212 (2003).

\bibitem{che86}
Chen, Y.G.:
\newblock Asymptotic behaviours of blowing-up solutions for finite difference analogue of {$u_t=u_{xx}+u^{1+\alpha}$}.
\newblock J. Fac. Sci. Univ. Tokyo Sect. IA Math 33, 541--574 (1986).

\bibitem{che92}
Chen, Y.G.:
\newblock Blow-up solutions to a finite difference analogue of {$u_t=\Delta
  u+u^{1+\alpha}$} in {$N$}-dimensional balls.
\newblock Hokkaido Math. J. 21(3), 447--474, (1992).

\bibitem{cho10}
Cho, C.H.:
\newblock A finite difference scheme for blow-up solutions of nonlinear wave equations.
\newblock Numer. Math. Theory Methods Appl. 3, 475--498 (2010).

\bibitem{cho07}
Cho, C.H.~, Hamada, S., Okamoto, H.:
\newblock On the finite difference approximation for a parabolic blow-up problem.
\newblock Japan J. Indust. Appl. Math. 24, 131--160 (2007).

\bibitem{dl00}
Deng, K.,~ Levine, H.A.:
\newblock The role of critical exponents in blow-up theorems: The sequel.
\newblock J. Math. Anal. Appl. 243(1), 85--126 (2000).

\bibitem{et84}
Eriksson, K., Thom\'{e}e, V.:
\newblock Galerkin methods for singular boundary value problems in one space
  dimension.
\newblock Math. Comp. 42(166), 345--367 (1984).

\bibitem{fuj66}
Fujita, H.:
\newblock On the blowing up of solutions of the {C}auchy problem for
  {$u_{t}=\Delta u+u^{1+\alpha }$}.
\newblock J. Fac. Sci. Univ. Tokyo Sect. I 13, 109--124 (1966).

\bibitem{ish10}
Ishiwata, M.:
\newblock On the asymptotic behavior of unbounded radial solutions for
  semilinear parabolic problems involving critical {S}obolev exponent.
\newblock J. Differential Equations 249(6), 1466--1482 (2010).

\bibitem{jes78}
Jespersen, D.:
\newblock Ritz--{G}alerkin methods for singular boundary value problems.
\newblock SIAM J. Numer. Anal. 15(4), 813--834, (1978).

\bibitem{lev90}
Levine, H.~A.:
\newblock The role of critical exponents in blowup theorems.
\newblock SIAM Rev. 32(2), 262--288 (1990).

\bibitem{nak76}
Nakagawa, T.:
\newblock Blowing up of a finite difference solution to 
  {$u_{t}=u_{xx}+u^{2}$}.
\newblock Appl. Math. Optim. 2, 337--350 (1976).

\bibitem{sai16w}
Saito, N.,~ Sasaki, T.:
\newblock  Blow-up of finite-difference solutions to nonlinear wave equations.
\newblock J. Math. Sci. Univ. Tokyo 23(1), 349--380 (2016).

\bibitem{sai16}
Saito, N. and Sasaki T.:
\newblock  Finite difference approximation for nonlinear Schr{\"{o}}dinger equations
with application to blow-up computation.
\newblock Japan J. Indust. Appl. Math. 33, 427--470 (2016).

\bibitem{tho06}
Thom{\'e}e, V.:
\newblock {G}alerkin finite element methods for parabolic problems, Second edition.
\newblock Springer Verlag, Berlin (2006).

\end{thebibliography}
\end{document}